\documentclass[reqno]{amsart}
\usepackage{amsmath}
\usepackage{amssymb}
\usepackage{amsfonts}
\usepackage{amsthm}
\usepackage{mathtools}
\usepackage{leftindex}
\usepackage{stmaryrd}
\usepackage{amsthm}
\usepackage{cite}
\usepackage{tikz} 
\usetikzlibrary{tqft,calc}
\usetikzlibrary{cd}
\usetikzlibrary{positioning}
\usepackage{enumitem}
\usepackage{graphicx}
\usepackage[left=2.5cm,right=2.5cm,top=3cm,bottom=3cm]{geometry}
\usepackage{tikz-cd}
\usepackage[hidelinks=true]{hyperref}
\usepackage{ulem}

\newcommand{\R}{\mathbb{R}}

\newcommand{\g}{\mathfrak{g}}

\newcommand{\greg}{\mathfrak{g}_{\mathrm{reg}}}

\newcommand{\rank}{\mathrm{rank}}

\renewcommand{\exp}{\mathrm{exp}}

\makeatletter
\@namedef{subjclassname@1991}{2020 Mathematics Subject Classification}
\makeatother

\numberwithin{equation}{section}

\newtheorem{theorem}{Theorem}[section]
\newtheorem{proposition}[theorem]{Proposition}
\newtheorem{corollary}[theorem]{Corollary}
\newtheorem{lemma}[theorem]{Lemma}

\newtheorem{conjecture}[theorem]{Conjecture}

\theoremstyle{definition}
\newtheorem{definition}[theorem]{Definition}
\newtheorem{example}[theorem]{Example}
\newtheorem{remark}[theorem]{Remark}
\newtheorem{question}{Question}

\newcommand {\G}{\mathcal G}

\newcommand{\AMT}{\mathsf{AMT}}

\renewcommand{\H}{\mathcal{H}}

\newcommand{\sll}[1]{\mkern-4mu\mathbin{/\mkern-5mu/}_{\mkern-4mu{#1}}}

\newcommand{\im}{\operatorname{im}}

\makeatletter
\@namedef{subjclassname@2020}{\textup{2020} Mathematics Subject Classification}
\makeatother

\newcommand{\sss}{\mathsf{s}}
\newcommand{\ttt}{\mathsf{t}}

\newcommand{\fp}[2]{\leftindex_{#1}\times_{#2}\,}

\newcommand{\tto}{\;\substack{\longrightarrow\\[-9pt] \longrightarrow}\;}

\begin{document}
	
	\title[Hamiltonian reduction in symplectic geometry and geometric representation theory]{Some incarnations of Hamiltonian reduction in symplectic geometry and geometric representation theory}
	
	\author[Peter Crooks]{Peter Crooks}
	\author[Xiang Gao]{Xiang Gao}
	\author[Mitchell Pound]{Mitchell Pound}
	\author[Casen Thompson]{Casen Thompson}
	\address{Department of Mathematics and Statistics\\ Utah State University \\ 3900 Old Main Hill \\ Logan, UT 84322, USA}
	\email[Peter Crooks]{peter.crooks@usu.edu}
	\email[Xiang Gao]{xiang.gao@usu.edu}
	\email[Mitchell Pound]{mitchell.pound@usu.edu}
	\email[Casen Thompson]{casen.thompson@usu.edu}

	\subjclass{53D17 (primary); 14L30 (secondary)} 
	\keywords{abelianization, Hamiltonian reduction, Moore--Tachikawa conjecture, symplectic groupoid}
	
	\begin{abstract}
		In this expository note, we give a self-contained introduction to some modern incarnations of Hamiltonian reduction. Particular emphasis is placed on applications to symplectic geometry and geometric representation theory. We thereby discuss abelianization in Hamiltonian geometry, reduction by symplectic groupoids, and the Moore--Tachikawa conjecture.
	\end{abstract}
	
	\maketitle
	\setcounter{tocdepth}{2}
	\tableofcontents
	\vspace{-10pt}

	
	\section{Introduction}
	
	\subsection{Motivation} Marsden--Weinstein reduction \cite{mar-wei:74} is arguably the first mathematically rigorous approach to symmetry reduction in classical mechanics\footnote{While the term ``Marsden--Weinstein reduction" is standard, one should acknowledge Meyer's independent approach to symplectic reduction \cite{mey:73}.}. Applications and generalizations of Marsden--Weinstein reduction continue to support and clarify new research avenues in symplectic geometry \cite{gui-jef-sja:02,cro-may:22,cro-roe:22,cro-wei:24,ler:95,mar-rat:86,sni-wei:83,sja-ler:91,mik-wei:88,bal-may:22,cro-may:241,sja:95,bie:97,mei:96,ale-mal-mei,mei-sja,ale-kos-mei,hur-jef-sja,hklr:87,ati-bot:84,alb:89,kir:84} and geometric representation theory \cite{bal:23,gin-kaz:24,gan-web,gan-gin:04,her-sch-sea:20,moo-tac:11,pan-toe-vaq-vez:13,saf:17,cro-may:24}. At the same time, we believe that many of the modern applications and generalizations of Marsden--Weinstein reduction do not appear in a single, accessible, self-contained manuscript. Our central purpose is to provide such a manuscript. Emphasis is placed on the following topics: abelianizing symplectic quotients via Gelfand--Cetlin data \cite{hof-lan:23,cro-wei:24,cro-wei:23,cro-wei:232,gui-ste:83}, generalizing Hamiltonian reduction to generic submanifolds/subvarieties \cite{cro-may:22,cro-may:241,bal-may:22}, and outlining a new perspective on the Moore--Tachikawa conjecture \cite{cro-may:24}. 
	
	\subsection{Context}
	One has the following paradigm for taking quotients in symplectic geometry. Suppose that a Lie group $G$ acts on a symplectic manifold $M$ by automorphisms of the symplectic structure. The $G$-action is called \textit{Hamiltonian} if it admits a \textit{moment map} $\mu:M\longrightarrow\g^*$, where $\g$ is the Lie algebra of $G$; see Definition \ref{Definition: Hamiltonian spaces}. Given $\xi\in\g^*$, the $G$-stabilizer $G_{\xi}\subset G$ acts on $\mu^{-1}(\xi)$. The \textit{Hamiltonian reduction of $M$ by $G$ at level $\xi\in\g^*$} is the symplectic manifold $M\sll{\xi}G\coloneqq\mu^{-1}(\xi)/G_{\xi}$, subject to whether $\mu^{-1}(\xi)/G_{\xi}$ is a manifold at all. Work of Marsden--Weinstein shows that $\mu^{-1}(\xi)/G_{\xi}$ is a manifold if $G_{\xi}$ acts freely and properly on $\mu^{-1}(\xi)$ \cite{mar-wei:74}. The following are some natural questions that arise in this context.
	
	\begin{question}\label{Question: Generalized Hamiltonian reduction}
		To what extent can the setup of Marsden--Weinstein reduction be generalized to yield new, interesting constructions in symplectic geometry?
	\end{question}
	
	\begin{question}\label{Question: Abelianization}
		It is typically easier to understand Marsden--Weinstein reductions by abelian stabilizer subgroups. To what extent can we abelianize the Marsden--Weinstein reductions of $M$ by $G$, in the sense of presenting each as a Hamiltonian reduction by an abelian group?
	\end{question}
	
	Various answers to Question \ref{Question: Generalized Hamiltonian reduction} have arisen in the fifty-plus years since Marsden--Weinstein reduction was formally introduced. One is Marsden and Ratiu's generalization to the case of Hamiltonian actions on Poisson manifolds \cite{mar-rat:86}. A second is Mikami and Weinstein's reduction by symplectic groupoids at singleton levels \cite{mik-wei:88}. Another is the \'{S}niatycki--Weinstein Poisson algebra associated to an arbitrary reduction level \cite{sni-wei:83}. In the case of a compact connected Lie group $G$, Sjamaar and Lerman remove the requirement that $G_{\xi}$ act freely on $\mu^{-1}(\xi)$; they show the topological space $M\sll{\xi}G=\mu^{-1}(\xi)/G_{\xi}$ to be a \textit{stratified symplectic space} for all $\xi\in\g^*$ \cite{sja-ler:91}. This framework is used in work of Guillemin--Jeffrey--Sjamaar to define a quotient procedure called \textit{symplectic implosion} \cite{gui-jef-sja:02}. A further interesting and useful variant of Marsden--Weinstein reduction is Lerman's \textit{symplectic cutting} procedure for Hamiltonian $\mathrm{U}(1)$-spaces \cite{ler:95}.
	
	In \cite{cro-may:22} and \cite{cro-may:241}, Mayrand and the first named author develop a theory of \textit{generalized Hamiltonian reduction} by symplectic groupoids along pre-Poisson submanifolds. This theory has versions for smooth manifolds, complex manifolds, complex analytic spaces, complex algebraic varieties, and affine schemes. It also witnesses all of the above-mentioned Hamiltonian reduction techniques as special cases. On the other hand, some of the most curious applications of generalized Hamiltonian reduction are to the \textit{Moore--Tachikawa conjecture} in geometric representation theory \cite{moo-tac:11}. In \cite{cro-may:24}, generalized Hamiltonian reduction is used to prove a scheme-theoretic version of this conjecture. The associated techniques are also shown to yield new topological quantum field theories in a category of affine Hamiltonian schemes. 
	
	We now address Question \ref{Question: Abelianization}. A first step is consider \textit{symplectic implosion}, as introduced by Guillemin--Jeffrey--Sjamaar \cite{gui-jef-sja:02}. Let $G$ be a compact connected Lie group with maximal torus $T\subset G$. The \textit{imploded cross-section} of $M$ is a stratified symplectic space $M_{\text{impl}}$ with a stratum-wise Hamiltonian action of $T$. A main result of \cite{gui-jef-sja:02} is that $M\sll{\xi}G$ and $M_{\text{impl}}\sll{\xi}T$ are isomorphic as stratified symplectic spaces for all $\xi$ in a fundamental Weyl chamber. As every Hamiltonian reduction of $M$ by $G$ is isomorphic to one taken at a level in the fundamental Weyl chamber, this result ``abelianizes" all Hamiltonian reductions of $M$ by $G$.
	
	Symplectic implosion allows one to replace the potentially non-abelian group $G$ with the abelian group $T$, at the expense of replacing $M$ with the possibly singular space $M_{\text{impl}}$. One might instead try to retain as much of $M$ as possible, at the expense of working with a higher-rank torus $\mathbb{T}_G$. Weitsman and the first named author accomplish this via \textit{Gelfand--Cetlin abelianizations} \cite{cro-wei:24}. This is largely based on the well-studied \textit{Gelfand-Cetlin systems} of Guillemin--Sternberg \cite{gui-ste:83}, and generalizations due to Hoffman--Lane \cite{hof-lan:23}. Very roughly speaking, one defines a \textit{Gelfand--Cetlin datum} to be a certain type of Hamiltonian $\mathbb{T}_G$-space structure on an open dense subset $\mathcal{U}\subset\g^*$. There exists an open subset $M^{\circ}\subset M$ with a Hamiltonian $\mathbb{T}_G$-action such that $M\sll{\xi}G\cong M^{\circ}\sll{\lambda(\xi)}\mathbb{T}_G$ as stratified symplectic spaces for all generic $\xi\in\g^*$, where $\lambda$ is the moment map for the $\mathbb{T}_G$-action on $\mathcal{U}$.    	 
	
	\subsection*{Acknowledgements}
	P.C. is supported by the National Science Foundation Grant DMS-2454103 and Simons Foundation Grant MPS-TSM-00002292. He also thanks the organizers of the workshop \textit{Symmetry and reduction in Poisson and related geometries}, held the University of Salerno in May 2024. All authors thank the Department of Mathematics and Statistics at Utah State University for providing an environment conducive to preparing this manuscript. We also thank the referee for constructive suggestions.
	
	\section{First cases of Hamiltonian reduction}
	In this section, we survey some of the most foundational cases of Hamiltonian reduction. We begin with preliminaries on Lie group actions and representation theory in Subsections \ref{Subsection: Group actions on manifolds} and \ref{Subsection: Representation theory}, respectively. Subsection \ref{Subsection: Linear Hamiltonian reduction} then describes a Hamiltonian reduction procedure for subspaces of a finite-dimensional symplectic vector space. To prepare for Hamiltonian reduction in broader contexts, Subsection \ref{Subsection: Symplectic manifolds} briefly reviews the basics of symplectic manifold theory. Subsection \ref{Subsection: Pre-symplectic reduction} then addresses the Hamiltonian reduction of a symplectic manifold along a pre-symplectic submanifold. In Subsection \ref{Subsection: Hamiltonian $G$-spaces}, we briefly review the definition of a Hamiltonian $G$-space. The previous two subsections are harnessed in Subsection \ref{Subsection: Marsden--Weinstein reduction}, where we recover Marsden--Weinstein reduction. We derive the so-called ``shifting trick" and universality of $T^*G$ in Subsection \ref{Subsection: Some identities}. Subsections \ref{Subsection: The orbit type} and \ref{Subsection: A generalization} then outline Sjamaar and Lerman's generalization of Hamiltonian reduction to stratified symplectic spaces.  
	
	\subsection{Group actions on manifolds}\label{Subsection: Group actions on manifolds} With the exception of Subsection \ref{Subsection: Reduction in symplectic vector spaces}, we take our base field to be $\Bbbk=\mathbb{R}$ or $\Bbbk=\mathbb{C}$. The reader should apply the adjectives \textit{real} (resp. \textit{complex}) and \textit{smooth} (resp. \textit{holomorphic}) where appropriate in the case $\Bbbk=\mathbb{R}$ (resp. $\Bbbk=\mathbb{C}$). All maps between manifolds, differential forms on manifolds, group actions on manifolds, and other constructions should be taken to be smooth or holomorphic, as appropriate. 
	
	Consider an equivalence relation $\sim$ on a manifold $M$. Let $\pi:M\longrightarrow M/\hspace{-3pt}\sim$ denote the associated quotient map. We say that $M/\hspace{-3pt}\sim$ \textit{is a manifold if} it carries a manifold structure for which $\pi$ is a submersion. In this case, the manifold structure on $M/\hspace{-3pt}\sim$ is unique. We are principally interested in equivalence relations arising from distributions and group actions.

	Given a Lie group $G$, a \textit{$G$-manifold} is a manifold endowed with a left action of $G$. Write $\g$ for the Lie algebra of $\g$, and $\exp:\g\longrightarrow G$ for the exponential map. The generating vector field of $x\in\g$ is defined by $$(x_M)_p\coloneqq\frac{d}{dt}\bigg\vert_{t=0}(\exp(-tx)\cdot p)$$ for all $p\in M$. The map $x\mapsto x_M$ defines a Lie algebra morphism from $\g$ to vector fields on $M$. On the other hand, write $$G\cdot p\coloneqq\{g\cdot p:g\in G\}\subset M\quad\text{and}\quad G_p\coloneqq\{g\in G:g\cdot p=p\}\subset G$$ for the $G$-orbit and $G$-stabilizer of $p\in M$, respectively. We have $$T_p(G\cdot p)=\{(x_M)_p:x\in\g\}\quad\text{and}\quad \g_p=\{x\in\g:(x_M)_p=0\},$$ where 
	$\g_p\subset\g$ denotes the Lie algebra of $G_p$. The action of $G$ on an invariant subset $N\subset M$ is called \textit{free} (resp. \textit{locally free}) if $G_p=\{e\}$ (resp. $\g_p=\{0\}$) for all $p\in N$. This action is called \textit{proper} if the map $$G\times N\longrightarrow N\times N,\quad (g,p)\mapsto (g\cdot p,p)$$ is a proper map of topological spaces. If the action of $G$ on $M$ is free and proper, then $M/G$ is a manifold.
	
	Two particularly important $G$-manifolds arise from representation theory. To this end, let $\mathrm{Ad}:G\longrightarrow\operatorname{GL}(\g)$ and $\mathrm{Ad}^*:G\longrightarrow\operatorname{GL}(\g^*)$ denote the adjoint and coadjoint representations of $G$, respectively. These induce the adjoint and coadjoint actions of $G$ on $\g$ and $\g^*$, respectively. One finds that $$\g_{x}=\{y\in\g:[x,y]=0\}\quad\text{and}\quad\g_{\xi}=\{y\in\g:\xi\circ[y,\cdot]=0\}$$ for all $x\in\g$ and $\xi\in\g^*$. The differentials of $\mathrm{Ad}$ and $\mathrm{Ad}^*$ at $e\in G$ are the adjoint and coadjoint representations of $\g$, respectively; we denote these by $\mathrm{ad}:\g\longrightarrow\mathfrak{gl}(\g)$ and $\mathrm{ad}^*:\g\longrightarrow\mathfrak{gl}(\g^*)$.
	
	\subsection{Representation theory}\label{Subsection: Representation theory}
	Suppose that $G$ is a compact connected real Lie group with Lie algebra $\g$. Let $T\subset G$ be a maximal torus with Lie algebra $\mathfrak{t}\subset\g$. On the other hand, consider the complex Lie algebras $\mathfrak{g}_{\mathbb{C}}\coloneqq\g\otimes_{\mathbb{R}}\mathbb{C}$ and $\mathfrak{t}_{\mathbb{C}}\coloneqq\mathfrak{t}\otimes_{\mathbb{R}}\mathbb{C}\subset\mathfrak{g}_{\mathbb{C}}$. Let $G_{\mathbb{C}}$ be a connected complex Lie group integrating $\g_{\mathbb{C}}$, and write $T_{\mathbb{C}}\subset G_{\mathbb{C}}$ for the maximal torus integrating $\mathfrak{t}_{\mathbb{C}}\subset\mathfrak{g}_{\mathbb{C}}$. One knows that $G_{\mathbb{C}}$ is a reductive affine algebraic group. It follows that $$\g_{\mathbb{C}}=\mathfrak{t}_{\mathbb{C}}\oplus\bigoplus_{\alpha\in\Phi}(\g_{\mathbb{C}})_{\alpha}$$ as $\mathfrak{t}_{\mathbb{C}}$-modules, where $\Phi\subset(\mathfrak{t}_{\mathbb{C}})^*$ is the set of roots and $$(\g_{\mathbb{C}})_{\alpha}\coloneqq\{x\in\mathfrak{g}_{\mathbb{C}}:[y,x]=\alpha(y)x\text{ for all }y\in\mathfrak{t}_{\mathbb{C}}\}.$$ A straightforward exercise reveals that $\alpha(x)\in i\mathbb{R}$ for all $x\in\mathfrak{t}$. Choose a collection $\Phi^{+}\subset\Phi$ of positive roots, and consider the cone
	$$\mathfrak{t}_{+}\coloneqq\{x\in\mathfrak{t}:-i\alpha(x)\geq 0\text{ for all }\alpha\in\Phi^{+}\}\subset\mathfrak{t}.$$ It turns out that $\mathfrak{t}_{+}$ is a fundamental domain for the adjoint action of $G$, i.e. the map
	$$\mathfrak{t}_{+}\longrightarrow\g/G,\quad x\mapsto G\cdot x$$ is a bijection.
	
	The compactness of $G$ allows us to choose a $G$-invariant inner product on $\g$. Note that the inner product induces a $G$-module isomorphism $\g\longrightarrow\g^*$. Let $\mathfrak{t}_{+}^*\subset\g^*$ denote the image of $\mathfrak{t}_{+}\subset\g$ under this isomorphism. It follows that $\mathfrak{t}_{+}^*$ is a fundamental domain for the coadjoint action of $G$. One calls $\mathfrak{t}^*_{+}$ a \textit{fundamental Weyl chamber}.
	
	\subsection{Linear Hamiltonian reduction}\label{Subsection: Reduction in symplectic vector spaces}\label{Subsection: Linear Hamiltonian reduction} Let $\Bbbk$ be a field. Consider a finite-dimensional $\Bbbk$-vector space $V$ and bilinear form $\omega:V\times V\longrightarrow\Bbbk$. We may form the linear map $$\omega^{\vee}:V\longrightarrow V^*,\quad v\mapsto\omega(v,\cdot).$$ One calls $\omega$ a \textit{symplectic form} if $\omega$ is skew-symmetric and non-degenerate, where the latter term means that $\omega^{\vee}$ is a vector space isomorphism. A \textit{symplectic vector space} is a finite-dimensional $\Bbbk$-vector space with a prescribed symplectic form. The dimension of a symplectic vector space is necessarily even. At the same time, there is exactly one isomorphism class of symplectic vector spaces in each even dimension. The isomorphism class of $2n$-dimensional symplectic vector spaces is represented by $\Bbbk^n\oplus\Bbbk^n$, with symplectic form
	$$(\Bbbk^n\oplus\Bbbk^n)\times (\Bbbk^n\oplus\Bbbk^n)\longrightarrow\Bbbk,\quad ((x_1,y_1),(x_2,y_2))\mapsto x_1^Ty_2-x_2^Ty_1.$$
	
	Suppose that $W\subset V$ is a subspace of a symplectic vector space $(V,\omega)$. One has the \textit{annihilator} $W^{\perp}\subset V^*$ of $W$, as well as the subspace $$W^{\omega}\coloneqq\{v\in V:\omega(v,w)=0\text{ for all }w\in W\}=(\omega^{\vee})^{-1}(W^{\perp})\subset V.$$ It follows that $\dim V=\dim W+\dim W^{\omega}$ and $(W^{\omega})^{\omega}=W$. The subspace $W$ is called \textit{isotropic} (resp. \textit{coisotropic}) if $W\subset W^{\omega}$ (resp. $W^{\omega}\subset W$).
	
	We may restrict $\omega$ to a bilinear form $$\omega\big\vert_W\coloneqq\omega\big\vert_{W\times W}:W\times W\longrightarrow\Bbbk.$$ While this restricted form is clearly skew-symmetric, non-degeneracy need not hold; consider the subspace $\Bbbk^n\oplus\{0\}$ of the symplectic vector space $\Bbbk^n\oplus\Bbbk^n$. Non-degeneracy holds if and only if the kernel of $(\omega\big\vert_W)^{\vee}:W\longrightarrow W^*$ is trivial. This kernel is $W\cap W^{\omega}$. It is also clear that $(\omega\big\vert_W)^{\vee}$ takes values in the subspace $$(W/(W\cap W^{\omega}))^*\cong(W\cap W^{\omega})^{\perp}\subset W^*.$$ These last few sentences imply that $\omega\big\vert_{W}$ descends to a symplectic form $$\overline{\omega\big\vert_{W}}:(W/(W\cap W^{\omega}))\times (W/(W\cap W^{\omega}))\longrightarrow\Bbbk.$$ In other words, $(W/(W\cap W^{\omega}),\overline{\omega\big\vert_{W}}))$ is a symplectic vector space.
	
	\begin{definition}[Linear Hamiltonian reduction]
		Let $W\subset V$ be a subspace of a symplectic vector space $V$. We define the \textit{Hamiltonian reduction of $W$ in $V$} to be the symplectic vector space $$\mathrm{Hred}_V(W)\coloneqq (W/(W\cap W^{\omega}),\overline{\omega\big\vert_{W}}).$$  
	\end{definition}
	
	\subsection{Symplectic manifolds}\label{Subsection: Symplectic manifolds}
	A \textit{symplectic manifold} consists of a manifold $M$ and \textit{symplectic form} on $M$, i.e. a $2$-form $\omega\in\Omega^2(M)$ that satisfies the following properties: $\mathrm{d}\omega=0$ and $(T_pM,\omega_p)$ is a symplectic vector space for all $p\in M$. The second condition can be reformulated as requiring $$\omega^{\vee}:TM\longrightarrow T^*M,\quad (p,v)\mapsto\omega_p(v,\cdot)$$ to be a bundle isomorphism. The notion of isomorphism between symplectic manifolds is clear, and called a \textit{symplectomorphism}. One has the following standard examples of symplectic manifolds.
	
	\begin{example}[Symplectic vector spaces]\label{Example: Symplectic vector spaces}
		Let $(V,\omega)$ be a symplectic vector space. We may regard $V$ as a manifold and $\omega$ as a $2$-form on this manifold. It is straightforward to check that $\mathrm{d}\omega=0$. Since $\omega_p=\omega$ for all $p\in V$, we conclude that $\omega$ is a symplectic form on $V$. It follows that symplectic vector spaces are symplectic manifolds.
	\end{example}
	
	Recall that a symplectic vector space is isomorphic to $\Bbbk^n\oplus\Bbbk^n$ for some $n\in\mathbb{Z}_{\geq 0}$, where the symplectic form on the latter is as in Subsection \ref{Subsection: Reduction in symplectic vector spaces}. On the other hand, Example \ref{Example: Symplectic vector spaces} implies that the symplectic vector space $\Bbbk^n\oplus\Bbbk^n$ is a symplectic manifold. It should therefore come as no surprise that each symplectic manifold is locally isomorphic to $\Bbbk^n\oplus\Bbbk^n$ for some $n\in\mathbb{Z}_{\geq 0}$; this is the \textit{Darboux theorem}. At the same time, one can use the dot product on $\Bbbk^n$ to identify $\Bbbk^n\oplus\Bbbk^n$ with $T^*(\Bbbk^n)$. The next example explains that the cotangent bundle of every manifold is canonically symplectic.
	
	\begin{example}[Cotangent bundles]\label{Example: Cotangent bundles}
		Let $M$ be a manifold. Consider the cotangent bundle projection $\pi:T^*M\longrightarrow M$. Define $\alpha\in\Omega^1(T^*M)$ by the property that $\alpha_{(p,\phi)}(v)=\phi(\mathrm{d}\pi_p(v))$ for all $(p,\phi)\in T^*M$ and $v\in T_{(p,\phi)}(T^*M)$. One calls $\alpha$ the \textit{tautological $1$-form} on $T^*M$. On the other hand, $-\mathrm{d}\alpha$ is a symplectic form on $T^*M$. The contangent bundle of every manifold is thereby a symplectic manifold.
	\end{example}
	
	\begin{example}[Opposites]\label{Example: Opposites}
		Let $M$ be a symplectic manifold with symplectic form $\omega$. A straightforward exercise reveals that $-\omega$ is also a symplectic form on $M$. The \textit{opposite} of $M$ is the symplectic manifold that results from equipping $M$ with $-\omega$; it is denoted $M^{\mathrm{op}}$.
	\end{example}
	
	\begin{example}[Products]\label{Example: Products}
		Let $M$ and $N$ be symplectic manifolds with respective symplectic forms $\omega_M\in\Omega^2(M)$ and $\omega_N\in\Omega^2(N)$. Let $\pi_M:M\times N\longrightarrow M$ and $\pi_N:M\times N\longrightarrow N$ denote the usual projection maps. One may verify that $\pi_M^*\omega_M+\pi_N^*\omega_N$ is a symplectic form on $M\times N$. A product of symplectic manifolds is thereby symplectic.
	\end{example}
	
	\begin{example}[Symplectic leaves]\label{Example: Symplectic leaves}
		Consider a manifold $M$ with a bivector field $\sigma\in\mathrm{H}^0(M,\wedge^2(TM))$. One calls $\sigma$ a \textit{Poisson bivector field} if $\{f_1,f_2\}\coloneqq\sigma(\mathrm{d}f_1,\mathrm{d}f_2)$ is a Poisson bracket on the structure sheaf of $M$. The pair $(M,\sigma)$ is then called a \textit{Poisson manifold}. In this case, consider the bundle morphism
		$$\sigma^{\vee}:T^*M\longrightarrow TM,\quad (p,\phi)\mapsto\sigma_p(\phi,\cdot).$$ It turns out that $\sigma^{\vee}$ is an isomorphism if and only if $\sigma^{\vee}=-(\omega^{\vee})^{-1}$ for a unique symplectic form $\omega$ on $M$. Conversely, a symplectic form $\omega$ on $M$ determines a Poisson bivector field $\sigma$ on $M$ by $\sigma^{\vee}=-(\omega^{\vee})^{-1}$. Symplectic manifolds are thereby those Poisson manifolds $(M,\sigma)$ for which $\sigma^{\vee}$ is a bundle isomorphism.
		
		Let $(M,\sigma)$ be a Poisson manifold. Note $\mathrm{im}(\sigma^{\vee})\subset TM$ is a possibly singular distribution on $M$. Despite this, $M$ admits a partition into integral leaves. Let $L\subset M$ be one such leaf. The Poisson bivector field $\sigma$ is tangent to $L$, and so determines a Poisson bivector field $\sigma_L\in H^0(L,\wedge^2(TL))$. It turns out that $\sigma_L^{\vee}:T^*L\longrightarrow TL$ is a bundle isomorphism. In light of the above, there exists a unique symplectic form $\omega_L\in\Omega^2(L)$ satisfying $\sigma_L^{\vee}=-(\omega_L^{\vee})^{-1}$. The integral leaves of $\mathrm{im}(\sigma^{\vee})$ are therefore called the \textit{symplectic leaves} of $M$.
		
		The following is a special case of the above. Let $G$ be a connected Lie group with Lie algebra $\g$. There is a unique Poisson bivector field $\sigma$ on $\g^*$ satisfying $\sigma_{\xi}^{\vee}(x)=\mathrm{ad}^*_x(\xi)$ for all $x\in\g$ and $\xi\in\g^*$; it is called the \textit{Lie-Poisson structure}. The symplectic leaves of this Poisson structure are precisely the coadjoint orbits of $G$. Given such an orbit $\mathcal{O}\subset\g^*$, note that $T_{\xi}\mathcal{O}=\{\xi\circ[x,\cdot]:x\in\g\}$ for all $\xi\in\mathcal{O}$. The symplectic form $\omega_{\mathcal{O}}$ on $\mathcal{O}$ is given by
		$$(\omega_{\mathcal{O}})_{\xi}(\xi\circ[x,\cdot],\xi\circ[y,\cdot])=\xi([x,y])$$ for all $\xi\in\mathcal{O}$ and $x,y\in\g$. 
		
		The connectedness of $G$ is not necessary for $\omega_{\mathcal{O}}$ to define a symplectic form on $\mathcal{O}$. This hypotheses is only used to ensure that the coadjoint orbits of $G$ constitute the symplectic leaves of $\g^*$.
	\end{example}
	
	\subsection{Pre-symplectic Hamiltonian reduction}\label{Subsection: Pre-symplectic reduction}	
	Let us consider a submanifold $N\subset M$ of a symplectic manifold $(M,\omega)$. In light of Subsection \ref{Subsection: Reduction in symplectic vector spaces}, it is reasonable that the symplectic vector spaces $\{\mathrm{Hred}_{T_pM}(T_pN)\}_{p\in N}$ might somehow form the tangent spaces to a ``Hamiltonian reduction of $N$ in $M$". A first step in this direction is to set
	$$(TN)^{\omega}\coloneqq \{(p,v)\in TM:p\in N\text{ and }v\in (T_pN)^{\omega_p}\}.$$ One may consider the distribution $\mathcal{D}_N\coloneqq TN\cap (TN)^{\omega}$ on $N$. It is then tempting to somehow take a ``quotient" of $N$ by this distribution. This endeavor would be most tractable if $\mathcal{D}_N$ were a regular distribution. Regularity amounts to the condition that $\dim(T_pN\cap (T_pN)^{\omega_p})$ not depend on $p\in N$. A submanifold $N\subset M$ with this property is called \textit{pre-symplectic}.
	
	In light of the above, let $N\subset M$ be a pre-symplectic submanifold of a symplectic manifold  $(M,\omega)$. One finds that the regular distribution $\mathcal{D}_N\subset TN$ is integrable. As such, it induces a foliation of $N$ into integral leaves. Write $N/\mathcal{D}_N$ for the associated leaf space, and $\pi:N\longrightarrow N/\mathcal{D}_N$ for the quotient map. If $N/\mathcal{D}_N$ is a manifold, the differential of $\pi$ induces an identification of $T_{[p]}(N/\mathcal{D}_N)$ with $T_pN/(T_pN\cap (T_pN)^{\omega_p})=\mathrm{Hred}_{T_pM}(T_pN)$ for all $p\in N$. We also know $\mathrm{Hred}_{T_pM}(T_pN)$ to be a symplectic vector space for all $p\in N$. The symplectic vector space structure thereby inherited by $T_{[p]}(N/\mathcal{D}_N)$ is seen not to depend on the choice of $p\in N$. This observation is part of the following stronger statement.
	
	\begin{theorem}\label{Theorem: General pre-symplectic reduction}
		Suppose that $N\subset M$ is a pre-symplectic submanifold of a symplectic manifold $(M,\omega)$. Consider the inclusion $j:N\longrightarrow M$ and quotient map $\pi:N\longrightarrow N/\mathcal{D}_N$. If $N/\mathcal{D}_N$ is a manifold, then there exists a unique symplectic form $\overline{\omega}$ on $N/\mathcal{D}_N$ satisfying $j^*\omega=\pi^*\overline{\omega}$.
	\end{theorem}
	
	\begin{definition}[Pre-symplectic Hamiltonian reduction]
		Assume that the hypotheses of Theorem \ref{Theorem: General pre-symplectic reduction} are satisfied. We define the \textit{Hamiltonian reduction of $N$ in $M$} to be the symplectic manifold
		$$\mathrm{Hred}_M(N)\coloneqq (N/\mathcal{D}_N,\overline{\omega}).$$
		
	\end{definition}
	
	\subsection{Hamiltonian $G$-spaces}\label{Subsection: Hamiltonian $G$-spaces}
	Let $(M,\omega)$ be a symplectic manifold. One defines the \textit{Hamiltonian vector field} of a function $f$ on $M$ by
	$X_f\coloneqq-(\omega^{\vee})^{-1}(\mathrm{d}f)$.  At the same time, suppose that $M$ is a $G$-manifold for a Lie group $G$ with Lie algebra $\g$. Recall that $\g^*$ carries the coadjoint action of $G$. Given a map $\mu:M\longrightarrow\g^*$ and element $x\in\g$, let $\mu^{x}:M\longrightarrow\Bbbk$ denote the result of pairing $\mu$ with $x$ pointwise.
	
	\begin{definition}\label{Definition: Hamiltonian spaces}
		Retain the notation and setup described above. 
		\begin{itemize}
			\item[\textup{(i)}] The map $\mu:M\longrightarrow\g^*$ is called a \textit{moment map} if it is $G$-equivariant and $X_{\mu^{x}}=x_M$ for all $x\in\g$.
			\item[\textup{(ii)}] One calls the $G$-action on $M$ \textit{Hamiltonian} if $\omega\in\Omega^2(M)^G$ and a moment map exists.
			\item[\textup{(iii)}] We call $(M,\omega,\mu)$ a \textit{Hamiltonian $G$-space} if the $G$-action on $M$ is Hamiltonian and $\mu:M\longrightarrow\g^*$ is a moment map.
		\end{itemize}
	\end{definition}
	
	The following are standard examples of Hamiltonian $G$-spaces.
	
	\begin{example}[Cotangent bundles of $G$-manifolds]\label{Example: Cotangent bundles of G-manifolds}
		Let $M$ be a $G$-manifold with action map $\mathcal{A}:G\times M\longrightarrow M$. Write $\mathcal{A}_g:M\longrightarrow M$ for the isomorphism given by acting by a fixed $g\in G$. It follows that $T^*M$ is a $G$-manifold via
		$$g\cdot(p,\phi)=(g\cdot p,\phi\circ((\mathrm{d}\mathcal{A}_{g})_p)^{-1}),\quad g\in G,\text { }(p,\phi)\in T^*M.$$ This action is called the \textit{cotangent lift} of the $G$-action on $M$. At the same time, Example \ref{Example: Cotangent bundles} explains that $T^*M$ is a symplectic manifold. The $G$-action is Hamiltonian with respect to the symplectic form on $T^*M$. A moment map is given by $$\mu:T^*M\longrightarrow\g^*,\quad \mu(p,\phi)(x)=-\phi((x_M)_p)$$ for all $(p,\phi)\in T^*M$ and $x\in\g$. 
	\end{example}
	
	\begin{example}[Poisson Hamiltonian actions]\label{Example: Poisson Hamiltonian actions} Let $(M,\sigma)$ be a Poisson manifold. Given a function $f$ on $M$, one defines the \textit{Hamiltonian vector field of $f$} by $X_f\coloneqq\sigma^{\vee}(\mathrm{df})$. Note that this generalizes the definition of a Hamiltonian vector field on a symplectic manifold. It also allows us to generalize Definition \ref{Definition: Hamiltonian spaces} from symplectic manifolds to Poisson manifolds; one simply replaces the $G$-invariance of $\omega$ with that of $\sigma$.
		
		Let $(M,\sigma,\mu)$ be a Hamiltonian $G$-space. If $G$ is connected, then $G$ preserves the symplectic leaves of $M$. The $G$-action on each leaf $L\subset M$ is Hamiltonian with respect to the symplectic form $\omega_L$ from Example \ref{Example: Symplectic leaves}. One finds that $\mu\big\vert_L:L\longrightarrow\g^*$ is a moment map.
		
		Recall the Lie-Poisson structure on $\g^*$ from Example \ref{Example: Symplectic leaves}. The coadjoint action of $G$ is Hamiltonian with respect to this Poisson structure, and admits the identity $\g^*\longrightarrow\g^*$ as a moment map. At the same time, Example \ref{Example: Symplectic leaves} tells us that each coadjoint orbit $\mathcal{O}\subset\g^*$ is symplectic. The $G$-action on $\mathcal{O}$ is Hamiltonian with moment map the inclusion $\mathcal{O}\longrightarrow\mathfrak{g}^*$.
		
	\end{example} 
	
	\begin{example}[Opposites]\label{Examples: Opposites Hamiltonian}
		Let $M$ be a Hamiltonian $G$-space with moment map $\mu:M\longrightarrow\g^*$. We may consider the opposite of the symplectic manifold $M$, as defined in Example \ref{Example: Opposites}. The action of $G$ on the opposite of $M$ is Hamiltonian with moment map $-\mu:M\longrightarrow\g^*$. Let us write $M^{\mathrm{op}}$ for this new Hamiltonian $G$-space.
	\end{example} 
	
	\begin{example}[Restrictions]\label{Example: Restrictions}
		Let $M$ be a Hamiltonian $G$-space with moment map $\mu:M\longrightarrow\g^*$. Consider a closed subgroup $H\subset G$ with Lie algebra $\mathfrak{h}\subset\mathfrak{g}$. Note that $M$ is an $H$-manifold via a restriction of the $G$-action. This $H$-action renders $M$ a Hamiltonian $H$-space, with moment map obtained by composing $\mu$ with the natural map $\g^*\longrightarrow\mathfrak{h}^*$.
	\end{example}
	
	\begin{example}[Products]\label{Example: Product Hamiltonian}
		Let $M$ and $N$ be Hamiltonian $G$-spaces with respective moment maps $\mu:M\longrightarrow\g^*$ and $\nu:N\longrightarrow\g^*$. Given Example \ref{Example: Products}, we may consider the product of the symplectic manifolds $M$ and $N$. The diagonal action of $G$ on this product is Hamiltonian with moment map $(p,q)\mapsto\mu(p)+\nu(q)$. We write $M\times N$ for this new Hamiltonian $G$-space.
	\end{example} 
	
	\begin{example}[A concrete example]\label{Example: A concrete example}
		Consider the unitary group $\mathrm{U}(n)$. Its Lie algebra $\mathfrak{u}(n)$ consists of the skew-Hermitian $n\times n$ matrices. Let $\mathcal{H}(n)$ denote the real $\mathrm{U}(n)$-module of Hermitian $n\times n$ matrices, where $\mathrm{U}(n)$ acts on $\mathcal{H}(n)$ by conjugation. The pairing $$\mathfrak{u}(n)\otimes_{\mathbb{R}}\mathcal{H}(n)\longrightarrow\mathbb{R},\quad x\otimes y\mapsto-\frac{i}{n}\mathrm{tr}(xy)$$ is non-degenerate and $\mathrm{U}(n)$-invariant. As such, it induces a $\mathrm{U}(n)$-module isomorphism $\mathfrak{u}(n)^*\cong\mathcal{H}(n)$. We thereby identify coadjoint orbits of $\mathrm{U}(n)$ with $\mathrm{U}(n)$-orbits in $\mathcal{H}(n)$. Example \ref{Example: Poisson Hamiltonian actions} now implies that the latter orbits are Hamiltonian $\mathrm{U}(n)$-spaces.
		
		Set $n=2$. Let $\mathcal{O}\subset\mathcal{H}(2)$ be the $\mathrm{U}(2)$-orbit of $$x\coloneqq\begin{bmatrix} 1 & 0\\ 0 & -1\end{bmatrix}\in\mathcal{H}(2).$$ Note that
		$$\mathrm{S}^1=\mathrm{U}(1)\longrightarrow\mathrm{U}(2),\quad t\mapsto\begin{bmatrix} t & 0 \\ 0 & t^{-1}\end{bmatrix}$$ is a Lie group isomorphism onto a closed subgroup of $\mathrm{U}(2)$. Let us thereby view $\mathrm{S}^1$ as a closed subgroup of $\mathrm{U}(2)$. By Example \ref{Example: Restrictions}, $\mathcal{O}$ is a Hamiltonian $\mathrm{S}^1$-space. On the other hand, one finds that $$\mathcal{O}=\left\{\begin{bmatrix} x & z \\ \bar{z} & -x\end{bmatrix}:x\in\mathbb{R},\text{ }z\in\mathbb{C},\text{ and }x^2+\vert z\vert^2=1\right\}.$$ It follows that
		$$\mathrm{S}^2\longrightarrow\mathcal{O},\quad (a,b,c)\mapsto\begin{bmatrix} c & a+bi \\ a-bi & -c\end{bmatrix}$$ is a diffeomorphism. We may endow $\mathrm{S}^2$ with the Hamiltonian $\mathrm{S}^1$-space structure for which this diffeomorphism is an $\mathrm{S}^1$-equivariant symplectomorphism intertwining moment maps. The $\mathrm{S}^1$-action on $\mathrm{S}^2$ is then a double-speed rotation of $\mathrm{S}^2$ about the vertical axis (``$c$-axis") in $\mathbb{R}^3$. Identifying $\mathrm{Lie}(\mathrm{S}^1)^*=\mathfrak{u}(1)^*$ with $\mathcal{H}(1)=\mathbb{R}$ as above, one finds the moment map to be $$\mu:\mathrm{S}^2\longrightarrow\mathbb{R},\quad (a,b,c)\mapsto c.$$ This is the ``height function" on $\mathrm{S}^2$. 
		
		
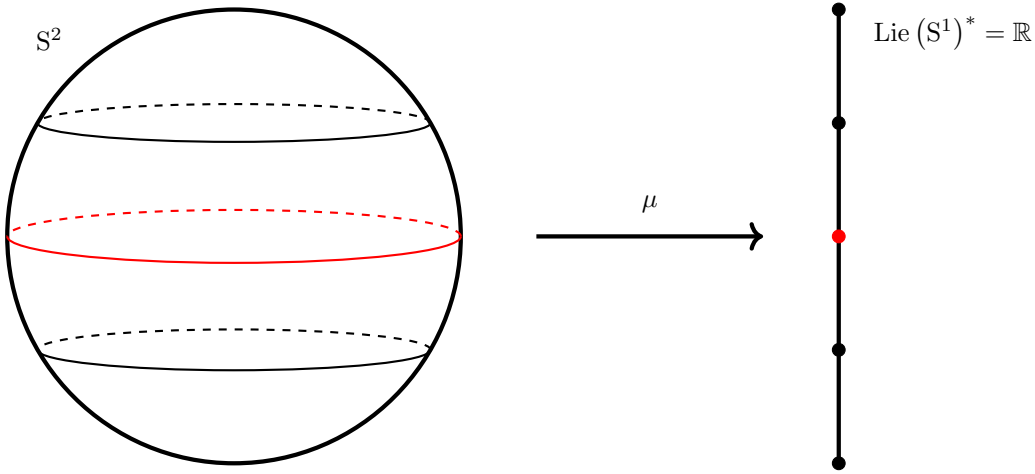
\vspace{10pt}
		
		\begin{center}
			\begin{tikzpicture}[x = 1cm, y = 1cm, scale= 1]
				\draw[ultra thick] (0,0) circle (3cm); 
				\node[above left= 2mm of {(-2,2.2)}] {$\mathrm{S}^2$};
				\draw[thick, red] (3,0) arc(0:-180:3cm and 0.35cm); 
				\draw[thick, dashed, red] (3,0) arc(0:180:3cm and 0.35cm);
				\draw[thick] (2.6,1.5) arc(0:-180:2.6cm and 0.25cm);
				\draw[thick, dashed] (2.6,1.5) arc(0:180:2.6cm and 0.25cm);
				\draw[thick] (2.6,-1.5) arc(0:-180:2.6cm and 0.27cm); 
				\draw[thick, dashed] (2.6,-1.5) arc(0:180:2.6cm and 0.27cm); 
				\draw[->, ultra thick] (4, 0) -- (7, 0); 
				\node[above = 2mm of {(5.5,0)}] {$\mu$};
				\draw[-, ultra thick] (8, -3) -- (8, 3); 
				\node[above right= 2mm of {(8.2,2.2)}] {$\mathrm{Lie} \left(\mathrm{S}^1\right)^*=\R$};
				\filldraw[red] (8, 0) circle (2.4pt);
				\filldraw[black] (8, 1.5) circle (2.4pt);
				\filldraw[black] (8, -1.5) circle (2.4pt);
				\filldraw[black] (8, 3) circle (2.4pt);
				\filldraw[black] (8, -3) circle (2.4pt);
			\end{tikzpicture}
		\end{center}    
	\end{example}
	
	\vspace{10pt}
	
	\subsection{Marsden--Weinstein reduction}\label{Subsection: Marsden--Weinstein reduction} Let $G$ be a Lie group with Lie algebra $\g$. Suppose that $(M,\omega,\mu)$ is a Hamiltonian $G$-space. Given $p\in M$, consider the differential $\mathrm{d}\mu_p:T_pM\longrightarrow\g^*$; it enjoys the following properties.
	
	\begin{lemma}\label{Lemma: Simple}
		Suppose that $p\in M$.
		\begin{itemize}
			\item[\textup{(i)}] We have $\ker(\mathrm{d}\mu_p)=(T_p(G\cdot p))^{\omega_p}$.
			\item[\textup{(ii)}] We have $\im(\mathrm{d}\mu_p)=(\g_p)^{\perp}$. 
		\end{itemize}
	\end{lemma} 
	
	\begin{proof}
		Suppose that $v\in T_pM$. We have $v\in \ker(\mathrm{d}\mu_p)$ if and only if $(\mathrm{d}\mu_p(v))(x)=0$ for all $x\in\g$. On the other hand, observe that evaluation at a fixed $x\in\g$ defines a linear map $\g^*\longrightarrow\Bbbk$. We conclude that
		\begin{equation}\label{Equation: Useful}(\mathrm{d}\mu_p(v))(x)=(\mathrm{d}\mu^{x})_p(v)=\omega_p(v,(x_M)_p)\end{equation} for all $x\in\g$, where the second equation follows from the definition of a moment map. This implies that $v\in \ker(\mathrm{d}\mu_p)$ if and only if $\omega_p(v,(x_M)_p)=0$ for all $x\in\g$, i.e. $v\in (T_p(G\cdot p))^{\omega_p}$. This proves (i).
		
		We now prove (ii). The inclusion $\im(\mathrm{d}\mu_p)\subset (\g_p)^{\perp}$ follows immediately from \eqref{Equation: Useful}. On the other hand, (i) implies that
		$$\dim(\im(\mathrm{d}\mu_p))=\dim(M)-\dim ((T_p(G\cdot p))^{\omega_p})=\dim (T_p(G\cdot p))=\dim\g-\dim\g_p=\dim((\g_p)^{\perp}).$$ These last two sentences imply (ii).
	\end{proof}
	
	Suppose that $\xi\in\g^*$. Since $\mu$ is $G$-equivariant, $\mu^{-1}(\xi)$ is invariant under the action of $G_{\xi}\subset G$ on $M$. The $G_{\xi}$-action on $\mu^{-1}(\xi)$ has the following properties.
	
	\begin{proposition}\label{Proposition: Somewhat simple}
		Suppose that $\xi\in\g^*$.
		\begin{itemize}
			\item[\textup{(i)}] The vector $\xi$ is a regular value of $\mu$ if and only if the action of $G_{\xi}$ on $\mu^{-1}(\xi)$ is locally free.
			\item[\textup{(ii)}] In the case of \textup{(i)}, $\mu^{-1}(\xi)$ is a pre-symplectic submanifold of $M$ and $$T_p(\mu^{-1}(\xi))\cap (T_p(\mu^{-1}(\xi)))^{\omega_p}=T_p(G_{\xi}\cdot p)$$ for all $p\in\mu^{-1}(\xi)$. 
		\end{itemize}
	\end{proposition} 
	
	\begin{proof}
		Lemma \ref{Lemma: Simple}(ii) implies that $\xi$ is a regular value if and only if $\g_p=\{0\}$ for all $p\in\mu^{-1}(\xi)$. At the same time, the $G$-equivariance of $\mu$ tells us that $\g_p\subset\g_{\xi}$ for all $p\in\mu^{-1}(\xi)$. It follows that $\xi$ is a regular value if and only if $(\g_{\xi})_p=\{0\}$ for all $p\in\mu^{-1}(\xi)$. This proves (i). 
		
		We now verify the tangent space identity in (ii). To this end, fix $p\in\mu^{-1}(\xi)$. Lemma \ref{Lemma: Simple}(i) implies that $T_p(\mu^{-1}(\xi))=T_p(G\cdot p)^{\omega_p}$. It therefore suffices to prove that $$T_p(G\cdot p)\cap (T_p(G\cdot p))^{\omega_p}=T_p(G_{\xi}\cdot p).$$  On the other hand, suppose that $x\in\g$. The $G$-equivariance of $\mu$ yields $\mathrm{d}\mu_p((x_M)_p)=(x_{\g^*})_{\xi}$. This gives
		\begin{align*}
			(x_M)_p\in T_p(G\cdot p)^{\omega_p} & \Longleftrightarrow \omega_p((x_M)_p,(y_M)_p)=0\hspace{5pt}\forall y\in\g\\ 	
			& \Longleftrightarrow (\mathrm{d}\mu_p((x_M)_p))(y)=0\hspace{5pt}\forall y\in\g\\
			& \Longleftrightarrow (x_{\g^*})_{\xi}=0\\
			& \Longleftrightarrow x\in\g_{\xi},
		\end{align*}
		where the second line comes from \eqref{Equation: Useful}.
		The assertion $T_p(G\cdot p)\cap (T_p(G\cdot p))^{\omega_p}=T_p(G_{\xi}\cdot p)$ follows immediately from this calculation.
		
		It remains only to prove that $\mu^{-1}(\xi)$ is pre-symplectic. Since the $G_{\xi}$-action on $\mu^{-1}(\xi)$ is locally free, $\dim (T_p(G_{\xi}\cdot\xi))=\dim (G_{\xi})$ for all $p\in\mu^{-1}(\xi)$. The identity $T_p(\mu^{-1}(\xi))\cap (T_p(\mu^{-1}(\xi)))^{\omega_p}=T_p(G_{\xi}\cdot p)$ then implies that the dimension of $T_p(\mu^{-1}(\xi))\cap (T_p(\mu^{-1}(\xi)))^{\omega_p}$ does not depend on $p\in\mu^{-1}(\xi)$. The proof of (ii) is therefore complete.
	\end{proof} 
	
	This result should be considered in the context of Subsection \ref{Subsection: Pre-symplectic reduction}. More precisely, suppose that the equivalent conditions of Proposition \ref{Proposition: Somewhat simple}(i) are satisfied. Part (ii) of this proposition implies that $\mu^{-1}(\xi)/\mathcal{D}_{\mu^{-1}(\xi)}=\mu^{-1}(\xi)/G_{\xi}$. 
	
	\begin{corollary}\label{Corollary: Reduction}
		Let $\xi\in\g^*$ be such that the action of $G_{\xi}$ on $\mu^{-1}(\xi)$ is locally free. Consider the inclusion $j:\mu^{-1}(\xi)\longrightarrow M$ and quotient $\pi:\mu^{-1}(\xi)\longrightarrow\mu^{-1}(\xi)/G_{\xi}$. If $\mu^{-1}(\xi)/G_{\xi}$ is a manifold, then there exists a unique symplectic form $\overline{\omega}_{\xi}$ on $\mu^{-1}(\xi)/G_{\xi}$ satisfying $j^*\omega=\pi^*\overline{\omega}_{\xi}$.
	\end{corollary}
	
	\begin{proof}
		This follows from Theorem \ref{Theorem: General pre-symplectic reduction}, Proposition \ref{Proposition: Somewhat simple}, and the discussion above. 
	\end{proof}
	
	\begin{definition}[Hamiltonian reduction by Lie groups]\label{Definition: Hamiltonian reduction by Lie groups}
		Suppose that $\xi\in\g^*$. Assume that the $G_{\xi}$-action on $\mu^{-1}(\xi)$ is locally free, and that $\mu^{-1}(\xi)/G_{\xi}$ is a manifold. The symplectic manifold $$M\sll{\xi}G\coloneqq \mathrm{Hred}_M(\mu^{-1}(\xi))=(\mu^{-1}(\xi)/G_{\xi},\overline{\omega}_{\xi})$$ is called the \textit{Hamiltonian reduction} of $M$ by $G$ at level $\xi$, and sometimes denoted $M\sll{\xi}G$. 
	\end{definition}
	
	To the authors' knowledge, $M\sll{\xi}G$ was first defined in Marsden and Weinstein's paper \cite{mar-wei:74} and Meyer's paper \cite{mey:73}. One includes an assumption sufficient to make $\mu^{-1}(\xi)/G_{\xi}$ a manifold, namely that $G_{\xi}$ acts freely and properly on $\mu^{-1}(\xi)$. This leads to the following construction.
	
	\begin{theorem}[Marsden--Weinstein reduction]\label{Theorem: Marsden-Weinstein reduction}
		Suppose that $\xi\in\g^*$. Consider the inclusion $j:\mu^{-1}(\xi)\longrightarrow M$ and quotient $\pi:\mu^{-1}(\xi)\longrightarrow\mu^{-1}(\xi)/G_{\xi}$. If $G_{\xi}$ acts freely and properly on $\mu^{-1}(\xi)$, then there exists a unique symplectic form $\overline{\omega}_{\xi}$ on $\mu^{-1}(\xi)/G_{\xi}$ satisfying $j^*\omega=\pi^*\overline{\omega}_{\xi}$.
	\end{theorem} 
	
	\subsection{Some identities in Hamiltonian reduction}\label{Subsection: Some identities} We now derive two ``identities" in the theory of Hamiltonian reduction; see Equations \ref{Equation: Shifting trick} and \ref{Equation: Cotangent identity}. Let $G$ be a Lie group with Lie algebra $\g$. Suppose that $(M,\omega,\mu)$ is a Hamiltonian $G$-space. Recall from Example \ref{Example: Poisson Hamiltonian actions} that the coadjoint orbit $\mathcal{O}\coloneqq G\cdot\xi\subset\g^*$ of $\xi\in\g^*$ is a Hamiltonian $G$-space. Examples \ref{Examples: Opposites Hamiltonian} and \ref{Example: Product Hamiltonian} then explain that $M\times\mathcal{O}^{\mathrm{op}}$ is a Hamiltonian $G$-space with moment map $$\nu:M\times\mathcal{O}^{\mathrm{op}}\longrightarrow\g^*,\quad (m,\zeta)\mapsto \mu(m)-\zeta.$$
	
	\begin{proposition}[Shifting trick]
		Suppose that $\xi\in\g^*$ and set $\mathcal{O}\coloneqq G\cdot\xi\subset\g^*$. Let $\nu:M\times\mathcal{O}^{\mathrm{op}}\longrightarrow\g^*$ be as defined above.
		\begin{itemize}
			\item[\textup{(i)}] The subgroup $G_{\xi}$ acts locally freely on $\mu^{-1}(\xi)$ if and only $G$ acts locally freely on $\nu^{-1}(0)$.
			\item[\textup{(ii)}] The quotient $\mu^{-1}(\xi)/G_{\xi}$ is a manifold if and only if $\nu^{-1}(0)/G$ is a manifold.
			\item[\textup{(iii)}] If the conditions in \textup{(i)} and \textup{(ii)} are satisfied, then the map $$\mu^{-1}(\xi)\longrightarrow M\times\mathcal{O},\quad p\mapsto (p,\xi)$$ descends to a symplectomorphism $M\sll{\xi} G\overset{\cong}\longrightarrow (M\times\mathcal{O}^{\mathrm{op}})\sll{0}G$.
		\end{itemize}
	\end{proposition}
	
	In short, one has the ``identity" \begin{equation}\label{Equation: Shifting trick} M\sll{\xi}G\cong(M\times\mathcal{O}^{\mathrm{op}})\sll{0}G\end{equation} for $\mathcal{O}=G\cdot\xi$. We conclude that Hamiltonian reductions at an arbitrary level can be obtained as Hamiltonian reductions at level $0$. This explains the nomenclature ``shifting trick"; one simply ``shifts" reduction at level $\xi$ to reduction at level zero.
	
	We now explain that $M$ can be obtained as a Hamiltonian reduction by $G$. To this end, note that $G\times G$ acts on $G$ by the formula
	$$(g_1,g_2)\cdot h=g_1hg_{2}^{-1},\quad (g_1,g_2)\in G\times G,\text{ }h\in G.$$ Example \ref{Example: Cotangent bundles of G-manifolds} implies that the cotangent lift of this $G\times G$-action to $T^*G$ renders $T^*G$ a Hamiltonian $G\times G$-space. On the other hand, we may consider the left trivialization of $T^*G$; this is the isomorphism
	$$T^*G\longrightarrow G\times\g^*,\quad (g,\phi)\mapsto (g,\phi\circ (\mathrm{d}L_g)_e)$$ of vector bundles over $G$, where $L_g:G\longrightarrow G$ is left multiplication by $g\in G$. We may equip $G\times\g^*$ with the Hamiltonian $G\times G$-space structure for which this vector bundle isomorphism is an isomorphism of Hamiltonian $G\times G$-spaces. This allows us to freely identify $T^*G$ and $G\times\g^*$ as Hamiltonian $G\times G$-spaces. The $G\times G$-action and moment map on $T^*G$ are then given by
	\begin{equation}\label{Equation: Product action}(g_1,g_2)\cdot (h,\xi)=(g_1hg_2^{-1},\mathrm{Ad}_{g_2}^*(\xi)),\quad (g_1,g_2)\in G\times G,\text{ }(h,\xi)\in G\times\g^*=T^*G\end{equation} and
	$$(\nu_1,\nu_2):T^*G=G\times\g^*\longrightarrow\g^*\oplus\g^*,\quad (g,\xi)\mapsto(\mathrm{Ad}_g^*(\xi),-\xi),$$ respectively.
	
	Let us restrict the Hamiltonian $G\times G$-space structure to the subgroup $G=\{e\}\times G\subset G\times G$. In this way, $T^*G$ is a Hamiltonian $G$-space with moment map $\nu_2$. The product $M\times T^*G$ is then a Hamiltonian $G$-space with moment map $$\phi:M\times T^*G\longrightarrow\g^*,\quad (p,(g,\xi))\mapsto \mu(p)-\xi.$$ 
	
	\begin{proposition}
		Let $\phi:M\times T^*G\longrightarrow\g^*$ be as defined above.
		\begin{itemize}
			\item[\textup{(i)}] The group $G$ acts freely on $\phi^{-1}(0)$.
			\item[\textup{(ii)}] The quotient $\phi^{-1}(0)/G$ is a manifold.
		\end{itemize}
	\end{proposition}
	
	This result allows us to form the symplectic manifold $(M\times T^*G)\sll{0} G$. At the same time, we may define the map
	$$\varphi:M\longrightarrow (M\times T^*G)\sll{0}G,\quad p\mapsto (p,(e,\mu(p))).$$ Note that $\phi$ is equivariant with respect to the $G$-action
	$$g\cdot [(p,(h,\xi))]\coloneqq [(p,(gh,\xi))],\quad g\in G,\text{ }[(p,(h,\xi))]\in(M\times T^*G)\sll{0}G$$ on $(M\times T^*G)\sll{0}G$. Let us also consider the map $$\psi:(M\times T^*G)\sll{0}G\longrightarrow\g^*,\quad [(p,(g,\xi))]\mapsto\mathrm{Ad}_g^*(\xi).$$
	
	\begin{proposition}\label{Proposition: Cotangent bundle identity}
		Let $\varphi:M\longrightarrow (M\times T^*G)\sll{0}G$ and $\psi:(M\times T^*G)\sll{0}G\longrightarrow\g^*$ be as defined above.
		\begin{itemize}
			\item[\textup{(i)}] The above-defined $G$-action on $(M\times T^*G)\sll{0}G$ is Hamiltonian with moment map $\psi$. 
			\item[\textup{(ii)}] The map $\varphi$ is an isomorphism of Hamiltonian $G$-spaces, where the Hamiltonian $G$-space structure on $(M\times T^*G)\sll{0}G$ comes from \textup{(i)}.
		\end{itemize}
	\end{proposition}
	
	This result may be loosely interpreted as the ``identity" \begin{equation}\label{Equation: Cotangent identity} M\cong(M\times T^*G)\sll{0}G\end{equation} for all Hamiltonian $G$-spaces $M$.
	
	\subsection{The orbit-type stratification for a proper action}\label{Subsection: The orbit type}
	Let $G$ be a connected Lie group. Note that $G$ acts on the set of its closed subgroups by conjugation. Let us write $[H]$ for the orbit of a closed subgroup $H\subset G$ under this action. On the other hand, let $M$ be a topological space with a proper $G$-action. Consider the $G$-invariant subset $$M_{[H]}\coloneqq\{p\in M:G_p\in[H]\}\subset M.$$ The connected components of $M_{[H]}$ are necessarily $G$-invariant. 
	
	\begin{definition}[Orbit-type stratifications]\label{Definition: Orbit-type stratifications}
		Retain the setup from above.
		\begin{itemize}
			\item[\textup{(i)}] A subset $X\subset M$ is called an \textit{orbit-type stratum} if $X$ is a connected component of $M_{[H]}$ for some closed subgroup $H\subset G$ with $M_{[H]}\neq\emptyset$.
			\item[\textup{(ii)}] A subset $Y\subset M/G$ is called an \textit{orbit-type stratum} if $Y=X/G$ for some orbit-type stratum $X\subset M$.
			\item[\textup{(iii)}] One defines the \textit{orbit-type stratifications} of $M$ and $M/G$ to be 
			$$M=\bigcup_{i\in I}M_i\quad\text{and}\quad M/G=\bigcup_{i\in I}(M_i/G),$$
			where $I$ indexes the orbit-type strata in $M$.
		\end{itemize}
	\end{definition}
	
	We refer the interested reader to \cite[Definition 1.7]{sja-ler:91} for a precise definition of a stratification. At the same time, we record the following properties of orbit-type stratifications.
	
	\begin{proposition}
		Assume that $M$ is a manifold carrying a proper $G$-action. Suppose that $i\in I$.
		\begin{itemize}
			\item[\textup{(i)}] The orbit-type stratum $M_i$ is a submanifold of $M$.
			\item[\textup{(ii)}] The orbit-type stratum $M_i/G$ is a manifold.
		\end{itemize}
	\end{proposition}
	
	\subsection{A generalization of Marsden--Weinstein reduction}\label{Subsection: A generalization}
	Let $(M,\omega,\mu)$ be a Hamiltonian $G$-space for a compact connected real Lie group $G$. One consequence of compactness is that $G_{\xi}$ acts properly on $\mu^{-1}(\xi)$ for all $\xi\in\g^*$. On the other hand, $G_{\xi}$ need not act freely on $\mu^{-1}(\xi)$ for a given $\xi\in\g^*$. This raises the question of whether Marsden--Weinstein reduction can be generalized to non-free actions of stabilizer subgroups on level sets of the moment map. Sjamaar and Lerman provide an affirmative answer in their seminal paper \cite{sja-ler:91}. The idea is to realize $\mu^{-1}(\xi)/G_{\xi}$ as a so-called \textit{stratified symplectic space}. We now provide some details.
	
	Fix $\xi\in\g^*$. The topological subspace $\mu^{-1}(\xi)\subset M$ carries an action of $G_{\xi}$. This fact allows us to consider the orbit-type stratifications $$\mu^{-1}(\xi)=\bigcup_{i\in I}\mu^{-1}(\xi)_i\quad\text{and}\quad \mu^{-1}(\xi)/G_{\xi}=\bigcup_{i\in I}(\mu^{-1}(\xi)_i/G_{\xi}).$$
	
	\begin{theorem}[Sjamaar--Lerman]
		Suppose that $\xi\in\g^*$ and $i\in I$.
		\begin{itemize}
			\item[\textup{(i)}] The subset $\mu(\xi)_i\subset M$ is a submanifold.
			\item[\textup{(ii)}] The quotient $\mu^{-1}(\xi)_i/G_{\xi}$ is a manifold.
			\item[\textup{(ii)}] There exists a unique symplectic form $(\overline{\omega}_{\xi})_i$ on $\mu^{-1}(\xi)_i/G_{\xi}$ satisfying $j^*\omega=\pi^*(\overline{\omega}_{\xi})_i$, where $j:\mu^{-1}(\xi)_i\longrightarrow M$ is the inclusion and $\pi:\mu^{-1}(\xi)_i\longrightarrow\mu^{-1}(\xi)_i/G_{\xi}$ is the quotient. 
		\end{itemize}
	\end{theorem}
	
	A great deal more is proved in \cite{sja-ler:91}. Among other things, Sjamaar and Lerman show $$M\sll{\xi}G\coloneqq \mu^{-1}(\xi)/G_{\xi}$$ to be a \textit{stratified symplectic space} \cite[Definition 1.12]{sja-ler:91}. This amounts to the orbit-type strata in $\mu^{-1}(\xi)/G_{\xi}$ enjoying several subtle and pertinent properties, in addition to being symplectic.
	
	\section{Abelianization}\label{Section: Abelianization}
	Suppose that $G$ is a compact connected real Lie group with Lie algebra $\g$. Fix a maximal torus $T\subset G$ with Lie algebra $\mathfrak{t}\subset G$. Let $\mathfrak{t}^*_{+}\subset\mathfrak{t}^*$ be the fundamental Weyl chamber constructed in Subsection \ref{Subsection: Representation theory}. There is a bijection from integral points in $\mathfrak{t}^*_{+}$ to isomorphism classes of finite-dimensional, simple, complex $G$-modules; it sends the isomorphism class of such a module to its highest weight. As finite-dimensional complex $G$-modules are semisimple, each is determined by the collection of highest weights (with multiplicity) that occur in a decomposition into simple $G$-modules. It follows that a finite-dimensional $G$-module can be recovered from its structure as a $T$-module. A Hamiltonian counterpart of this statement is that every Hamiltonian reduction of a Hamiltonian $G$-space should be realizable as a Hamiltonian reduction of a related space by $T$. This perspective is fully realized via Guillemin, Jeffrey, and Sjamaar's notion of \textit{symplectic implosion}; see Subsection \ref{Subsection: Symplectic implosion}. It constitutes a kind of ``abelianization" of symplectic quotients by $G$. In Subsection \ref{Subsection: Gelfand--Cetlin data}, we use Gelfand--Cetlin theory to outline an alternative approach to abelianization.

	\subsection{Symplectic implosion}\label{Subsection: Symplectic implosion} Retain the objects and notation introduced at the beginning of this section. Let $(M,\omega,\mu)$ be a connected Hamiltonian $G$-space. Define an equivalence relation on $\mu^{-1}(\mathfrak{t}^*_{+})$ by $p\sim q$ if $q=g\cdot p$ for some $g\in [G_{\mu(p)},G_{\mu(p)}]$, and let $$M_{\text{impl}}\coloneqq\mu^{-1}(\mathfrak{t}^*_{+})/\hspace{-3pt}\sim$$ denote the topological space quotient. It is clear that $\mu\big\vert_{\mathfrak{t}_{+}^*}$ is invariant under this equivalence relation. In particular, $\mu\big\vert_{\mathfrak{t}_{+}^*}$ descends to a continuous map $\mu_{\text{impl}}:M_{\text{impl}}\longrightarrow\mathfrak{t}_{+}^*$.
	
	\begin{definition}
		The \textit{imploded cross-section} of $M$ is the pair $(M_{\text{impl}},\mu_{\text{impl}})$.
	\end{definition}
	
	To realize some richer structure on $(M_{\text{impl}},\mu_{\text{impl}})$, let $\Sigma$ denote the set of open faces of $\mathfrak{t}^*_{+}$. Define a partial order on $\Sigma$ by $\sigma\leq\tau$ if $\sigma\subset\overline{\tau}$. One finds that $G_{\xi}=G_{\eta}$ for all $\xi,\eta$ in a given face of $\mathfrak{t}^*_{+}$. This allows one to define $G_{\sigma}$ as the $G$-stabilizer of any point in a face $\sigma$ of $\mathfrak{t}^*_{+}$. We have the set-theoretic disjoint union
	$$M_{\text{impl}}=\bigcup_{\sigma\in\Sigma}\left(\mu^{-1}(\sigma)/[G_{\sigma},G_{\sigma}]\right).$$ For each $\sigma\in\Sigma$, the orbit-type strata of $\mu^{-1}(\sigma)/[G_{\sigma},G_{\sigma}]$ are naturally symplectic manifolds. These last two sentences turn out to underlie a more subtle fact: $M_{\text{impl}}$ is a stratified symplectic space. On the other hand, $\mu^{-1}(\mathfrak{t}_{+}^*)$ is a $T$-invariant subset of $M$. We also observe that $T$ normalizes $[G_{\sigma},G_{\sigma}]$ for all $\sigma\in\Sigma$. It follows that the $T$-action on $\mu^{-1}(\mathfrak{t}_{+}^*)$ descends to a $T$-action on $M_{\text{impl}}$. The latter action preserves the symplectic strata of $M_{\text{impl}}$, as well as the symplectic form on each stratum. Restricting $\mu_{\text{impl}}$ to each stratum yields a moment map for the $T$-action on that stratum. Given $\xi\in\mathfrak{t}_{+}^*$, this fact allows one to realize the quotient topological space $$M_{\text{impl}}\sll{\xi}T\coloneqq\mu_{\text{impl}}^{-1}(\xi)/T$$ as a stratified symplectic space. It enjoys the following relation to the stratified symplectic space $M\sll{\xi}G$; see \cite[Theorem 3.4]{gui-jef-sja:02}.
	
	\begin{theorem}\label{Theorem: Implosion}
		If $\xi\in\mathfrak{t}_{+}^*$, then the natural map $\mu^{-1}(\xi)\longrightarrow\mu_{\mathrm{impl}}^{-1}(\xi)$ descends to an isomorphism
		$$M\sll{\xi}G\overset{\cong}\longrightarrow M_{\mathrm{impl}}\sll{\xi}T$$ of stratified symplectic spaces.
	\end{theorem}  
	
	A straightforward exercise reveals that $M\sll{\xi}G\cong M\sll{\mathrm{Ad}_g^*(\xi)}G$ for all $\xi\in\g^*$ and $g\in G$. Let us also recall that $\mathfrak{t}_{+}^*$ is a fundamental domain for the coadjoint action of $G$ on $\g^*$. It follows that any Hamiltonian reduction of $M$ by $G$ is isomorphic to $M\sll{\xi}G$ for a suitable $\xi\in\mathfrak{t}_{+}^*$. In this way, $M_{\text{impl}}$ allows one to systematically ``abelianize" the Hamiltonian reductions of $M$ by the potentially non-abelian group $G$.
	
	\begin{remark}[Universal imploded cross-sections]
		As discussed in Subsection \ref{Subsection: Some identities}, $T^*G$ is a Hamiltonian $G\times G$-space. Restriction to the subgroup $G=\{e\}\times G\subset G\times G$ renders $T^*G$ a Hamiltonian $G$-space. We may form the imploded cross-section $(T^*G)_{\text{impl}}$ with respect to this Hamiltonian $G$-space structure. At the same time, the action of $G=G\times\{e\}\subset G\times G$ on $T^*G$ induces an action of $G$ on $(T^*G)_{\text{impl}}$. The latter action commutes with the $T$-action on $(T^*G)_{\text{impl}}$, and preserves the symplectic strata of $(T^*G)_{\text{impl}}$. It follows that $M\times (T^*G)_{\text{impl}}$ is a stratified symplectic space with a stratum-wise Hamiltonian $G$-action. These considerations allow one to make sense of $(M\times (T^*G)_{\text{impl}})\sll{0}G$ as a stratified symplectic space with a stratum-wise Hamiltonian $T$-action. As such, $(M\times (T^*G)_{\text{impl}})\sll{0}G$ is canonically isomorphic to $M_{\text{impl}}$; see \cite[Theorem 4.9]{gui-jef-sja:02}. This fact justifies calling $(T^*G)_{\text{impl}}$ the \textit{universal imploded cross-section} for $G$.
	\end{remark}
	
	\subsection{Gelfand--Cetlin data}\label{Subsection: Gelfand--Cetlin data}
	Despite being a powerful approach to abelianizing Hamiltonian reductions, symplectic implosion features the following cost: one must replace a Hamiltonian $G$-space $M$ with its imploded cross-section $M_{\text{impl}}$, a potentially singular $T$-space. It is natural to wonder if $M$ carries a Hamiltonian action of a higher-rank torus $\mathbb{T}_G$, in such a way that Hamiltonian reductions of $M$ by $G$ are canonically isomorphic to Hamiltonian reductions of $M$ by $\mathbb{T}_G$. We now outline an approach to abelianization in this spirit of this question; it features the Guillemin and Sternberg's \textit{Gelfand--Cetlin systems} \cite{gui-ste:83}, as well as Hoffman and Lane's recent generalization thereof \cite{hof-lan:23}.     
	
	We first introduce some Lie-theoretic data. Consider the quantities $$\mathrm{rk}\hspace{2pt}G\coloneqq\rank\hspace{2pt}G,\quad\mathrm{u}_{G}\coloneqq\frac{1}{2}(\dim G-\mathrm{rk}\hspace{2pt}G),\quad\text{and}\quad\mathrm{b}_{G}\coloneqq\frac{1}{2}(\dim G+\mathrm{rk}\hspace{2pt}G),$$ as well as the three tori
	$$T_G\coloneqq\mathrm{U}(1)^{\mathrm{rk}\hspace{2pt}G},\quad \mathcal{T}_G\coloneqq\mathrm{U}(1)^{\mathrm{u}_G},\quad\text{and}\quad\mathbb{T}_G\coloneqq T_G\times\mathcal{T}_G\cong\mathrm{U}(1)^{\mathrm{b}_{G}}.$$ On the other hand, use multiplication by $-i$ to identify $\mathfrak{u}(1)=i\mathbb{R}$ with $\mathbb{R}$. The Lie algebras of $T_G$, $\mathcal{T}_G$, and $\mathbb{T}_G$ are thereby $\mathbb{R}^{\mathrm{rk}\hspace{2pt}G}$, $\mathbb{R}^{\mathrm{u}_G}$, and $\mathbb{R}^{\mathrm{b}_G}$, respectively. Let us also consider the locus $$\g^*_{\text{reg}}\coloneqq\{\xi\in\g^*:\dim\g_{\xi}=\mathrm{rk}\hspace{2pt}G\}.$$
	
	\begin{definition}
		A \textit{Gelfand--Cetlin datum} for $G$ is a pair $(\lambda,\mathcal{U})$ of a continuous map $\lambda=(\lambda_1,\ldots,\lambda_{\mathrm{b}_G}):\g^*\longrightarrow\mathbb{R}^{\mathrm{b}_G}$ and open dense subset $\mathcal{U}\subset\g^*$ that satisfy the following properties:
		\begin{itemize}
			\item[\textup{(i)}] $\lambda_1,\ldots,\lambda_{\mathrm{rk}_G}$ are $G$-invariant on $\g^*$ and smooth on $\g^*_{\text{reg}}$;
			\item[\textup{(ii)}] for all $\xi\in\g^*_{\text{reg}}$, $\{(\mathrm{d}\lambda_1)_{\xi},\ldots,(\mathrm{d}\lambda_{\mathrm{rk}_G})_{\xi}\}$ is a $\mathbb{Z}$-basis of the lattice $\frac{1}{2\pi}\ker(\exp\big\vert_{\g_{\xi}}:\g_{\xi}\longrightarrow G_{\xi})$;
			\item[\textup{(iii)}] $\mathcal{U}\subset\g^*_{\text{reg}}$;
			\item[\textup{(iv)}] $\lambda\big\vert_{\mathcal{U}}:\mathcal{U}\longrightarrow\mathbb{R}^{\mathrm{b}_G}$ is a smooth submersion and moment map for a Poisson Hamiltonian action of $\mathbb{T}_G$ on $\mathcal{U}$;
			\item[\textup{(v)}] $\lambda\big\vert_{\mathcal{U}}:\mathcal{U}\longrightarrow\lambda(\mathcal{U})$ is a principal bundle for $\mathcal{T}_G$;
			\item[\textup{(vi)}] if $M$ is a symplectic Hamiltonian $G$-space with moment map $\mu:M\longrightarrow\g^*$, then $\lambda\circ(\mu\big\vert_{\mu^{-1}(\mathcal{U})}):\mu^{-1}(\mathcal{U})\longrightarrow\mathbb{R}^{\mathrm{b}_G}$ is a moment map for a Hamiltonian action of $\mathbb{T}_G$-action on $\mu^{-1}(\mathcal{U})$.
		\end{itemize}
	\end{definition}
	
	Let $(\lambda,\mathcal{U})$ be a Gelfand--Cetlin datum for $G$. Given a Hamiltonian $G$-space $M$ with moment map $\mu:M\longrightarrow\g^*$, $\mu^{-1}(\mathcal{U})$ is a Hamiltonian $\mathbb{T}_{G}$-space with moment map $\lambda\circ(\mu\big\vert_{\mu^{-1}(\mathcal{U})}):\mu^{-1}(\mathcal{U})\longrightarrow\mathbb{R}^{\mathrm{b}_G}$. It is therefore reasonable to compare Hamiltonian reductions of $\mu^{-1}(\mathcal{U})$ by $\mathbb{T}_G$ to those of $M$ by $G$; each is a stratified symplectic space. This is the impetus for the following result; see \cite{cro-wei:24} for more details.
	
	\begin{theorem}[C.--Weitsman]
		Let $(\lambda,\mathcal{U})$ be a Gelfand--Cetlin datum for $G$. If $M$ is a Hamiltonian $G$-space and $\xi\in\mathcal{U}$, then the inclusion $\mu^{-1}(\xi)\longrightarrow(\lambda\circ\mu)^{-1}(\lambda(\xi))$ descends to an isomorphism
		$$M\sll{\xi}G\cong\mu^{-1}(\mathcal{U})\sll{\lambda(\xi)}\mathbb{T}_G$$ of stratified symplectic spaces.
	\end{theorem}
	
	\begin{example}[The classical Gelfand--Cetlin system]
		Suppose that $G=\mathrm{U}(n)$ is the group of $n\times n$ unitary matrices. Its Lie algebra is the $\mathbb{R}$-vector space $\mathfrak{u}(n)$ of $n\times n$ skew-Hermitian matrices. Identify the $\mathrm{U}(n)$-modules $\mathfrak{u}(n)^*$ and $\mathcal{H}(n)$ in the manner of Example \ref{Example: A concrete example}. In this case, we use the slightly different pairing $$\mathfrak{u}(n)\otimes_{\mathbb{R}}\mathcal{H}(n)\longrightarrow\mathbb{R},\quad x\otimes y\mapsto-{i}\mathrm{tr}(xy)$$ to induce the identification.
		
		Suppose that $\xi\in\mathfrak{u}(n)^*=\mathcal{H}(n)$. Given $i\in\{0,\ldots,n-1\}$, the $(n-i)\times (n-i)$ submatrix in the bottom-right corner of $\xi$ has real eigenvalues $\lambda_{i1}(\xi)\geq\cdots\geq\lambda_{i(n-i)}(\xi)$. This observation defines $\mathrm{b}_{\mathrm{U}(n)}=\frac{n(n+1)}{2}$ continuous functions $\lambda_{ij}:\mathfrak{u}(n)^*\longrightarrow\mathbb{R}$, where $i\in\{0,\ldots,n-1\}$ and $j\in\{1,\ldots,n-i\}$. Let us define $\lambda:\mathfrak{u}(n)^*\longrightarrow\mathbb{R}^{\mathrm{b}_G}$ and $\mathcal{U}\subset\mathfrak{u}(n)^*$ by  $$\lambda\coloneqq(\lambda_{01},\ldots,\lambda_{0n},\lambda_{11},\ldots,\lambda_{1(n-1)},\ldots,\lambda_{(n-1)1})$$ and $$\mathcal{U}\coloneqq\{\xi\in\mathfrak{u}(n)^*:\lambda_{ij}(\xi)>\lambda_{(i+1)j}(\xi)>\lambda_{i(j+1)}(\xi)\text{ for all }i\in\{0,\ldots,n-2\}\text{ and }j\in\{1,\ldots,n-i-1\}\},$$ respectively. The pair $(\lambda,\mathcal{U})$ is a Gelfand--Cetlin datum for $\mathrm{U}(n)$; see \cite[Subsection 3.5]{cro-wei:24}.
		
		This example warrants a few comments. Let $\mathcal{O}\subset\mathfrak{u}(n)^*$ be a coadjoint orbit of $\operatorname{U}(n)$ satisfying $\mathcal{O}\subset\mathfrak{u}(n)^*_{\text{reg}}$. It follows that $\dim\mathcal{O}=n(n-1)$. On the other hand, there are $\frac{n(n-1)}{2}$ functions of the form $\lambda_{ij}$ with $i\in\{1,\ldots,n-1\}$ and $j\in\{1,\ldots,n-i\}$. The restrictions of these functions to $\mathcal{O}\cap\mathcal{U}$ constitute a completely integrable system with respect to the symplectic structure on the open subset of $\mathcal{O}\cap\mathcal{U}\subset\mathcal{O}$. Guillemin and Sternberg call this system the \textit{Gelfand--Cetlin system} on $\mathcal{O}$ \cite{gui-ste:83}. Among other things, it affords a geometric quantization of $\mathcal{O}$ in a real polarization \cite{gui-ste:83}.
	\end{example}
	
	\begin{example}[General examples of Gelfand--Cetlin data]
		It turns out that every compact connected Lie group admits Gelfand--Cetlin data. This fact follows from the work of Hoffman--Lane \cite{hof-lan:23}; see \cite[Section 3.2]{cro-wei:24} for further details. Hoffman and Lane use \textit{toric degenerations} as a technical tool, thereby building on results of Harada--Kaveh \cite{har-kav:15}.
	\end{example}
	
	\begin{example}[Further applications of Gelfand--Cetlin data]
		In \cite{cro-wei:232}, Weitsman and the first named author construct the \textit{double Gelfand--Cetlin system} on $T^*\mathrm{U}(n)$. The resulting real polarization yields a geometric quantization of $T^*\mathrm{U}(n)$. Hoffman and Lane extent these ideas to arbitrary Lie type \cite[Example 6.14]{hof-lan:23}. On the other hand, Weitsman and the first author use Gelfand--Cetlin data to prove a non-abelian version of the Duistermaat--Heckman theorem \cite{cro-wei:23}.
	\end{example}
	
	\section{Generalized Hamiltonian reduction}
	This section surveys a series of manuscripts written by Mayrand and the first named author \cite{cro-may:22,cro-may:241,cro-may:24}. Subsection \ref{Subsection: Symplectic groupoids and their actions} begins with a brief overview of symplectic groupoids and their Hamiltonian actions. In Subsection \ref{Subsection: Pre-Poisson reduction}, we define a \textit{generalized Hamiltonian reduction} procedure for Hamiltonian actions of symplectic groupoids; it encompasses many of the procedures discussed earlier in this manuscript. Generalized Hamiltonian reduction also has implications for the Moore--Tachikawa conjecture \cite{moo-tac:11}, as explained in Subsection \ref{Subsection: The Moore--Tachikawa conjecture}. In the interest of further elucidating this conjecture, Subsection \ref{Subsection: A scheme-theoretic version} outlines a scheme-theoretic version of generalized Hamiltonian reduction. The importance of scheme-theoretic reduction to the Moore--Tachikawa conjecture is outlined in Subsection \ref{Subsection: MT via SG}. This is explained in the context of a broader theme: abelian symplectic groupoids and Morita equivalences yield topological quantum field theories.  
	
	\subsection{Symplectic groupoids and their actions}\label{Subsection: Symplectic groupoids and their actions}
	Consider a groupoid object $\G\tto X$ in the category of smooth manifolds (resp. complex manifolds, complex algebraic varieties). We call $\G\tto X$ a real Lie groupoid (resp. complex Lie groupoid, algebraic Lie groupoid) if its source and target morphisms are smooth submersions (resp. holomorphic submersions, smooth morphisms of varieties). One calls $\G\tto X$ a \textit{real symplectic groupoid} (resp. 
	\textit{complex symplectic groupoid}, \textit{algebraic symplectic groupoid}) if $\G$ is equipped with a real symplectic (resp. holomorphic symplectic, algebraic symplectic) form for which the graph of groupoid multiplication is Lagrangian in $\G\times\G\times\G^{\text{op}}$. In what follows, we use the term \textit{symplectic groupoid} to encompass these three cases. The correct specializations to the smooth, holomorphic, and algebraic categories will be clear from context. 
	
	\begin{example}[Cotangent groupoids]\label{Example: Cotangent groupoids}
		Let $G$ be a Lie group with Lie algebra $\g$. Identify $T^*G$ with $G\times\g^*$ by left-trivializing the former. Define $\sss:T^*G\longrightarrow\mathfrak{g}^*$ and $\ttt:T^*G\longrightarrow\mathfrak{g}^*$ by $\sss(g,\xi)=\mathrm{Ad}_g^*(\xi)$ and $\ttt(g,\xi)=\xi$, respectively. It turns out that $\sss$ and $\ttt$ are the source and target morphisms of a symplectic groupoid structure on $T^*G\tto\g^*$, where $T^*G$ carries the symplectic form in Example \ref{Example: Cotangent bundles}. Groupoid multiplication is given by $(g,\xi)\cdot (h,\eta)\coloneqq (gh,\eta)$ if $\xi=\mathrm{Ad}_h^*(\eta)$.
	\end{example}
	
	\begin{example}[The universal centralizer]\label{Example: Universal centralizer}
		Let $G$ be a connected complex semisimple affine algebraic group with Lie algebra $\g$. Use the Killing form and left trivialization to identify $T^*G$ with $G\times\g$ and $\g^*$ with $\g$. In light of Example \ref{Example: Cotangent groupoids}, the source and target of the cotangent groupoid $T^*G\tto\g$ are $\sss(g,x)=\mathrm{Ad}_g(x)$ and $\ttt(g,x)=x$, respectively. Note that
		$$\mathcal{I}\coloneqq\{(g,x)\in T^*G:\mathrm{Ad}_g(x)=x\}$$ is a subgroupoid of $T^*G$, on which the source and target morphisms coincide. It is thereby a group scheme $\mathcal{I}\longrightarrow\g$ with fiber $G_x$ over $x\in\g$. 
		
		Consider the regular locus $$\greg\coloneqq\{x\in\g:\dim\g_x=\mathrm{rk}\hspace{2pt}\g\}.$$ At the same time, call $(e,h,f)\in\g^{\times 3}$ an \textit{$\mathfrak{sl}_2$-triple} if $[e,f]=h$, $[h,e]=2e$, and $[h,f]=-2f$. An $\mathfrak{sl}_2$-triple $(e,h,f)\in\mathfrak{g}^{\times 3}$ is called \textit{principal} if $e,h,f\in\greg$. Such triples exist, and any two are conjugate \cite{kos:59}. Given a principal $\mathfrak{sl}_2$-triple $(e,h,f)\in\mathfrak{g}^{\times 3}$, consider the $\mathrm{rk}\g\hspace{2pt}$-dimensional affine subspace $\mathcal{S}\coloneqq e+\g_f\subset\g$. It turns out that $\mathcal{S}\subset\greg$, and that $\mathcal{S}$ is a fundamental domain for the adjoint action of $G$ on $\greg$ \cite{kos:63}. 
		
		The pullback
		$$\begin{tikzcd}
			\mathcal{J}\arrow[r] \arrow[d] & \mathcal{I} \arrow[d] \\
			\mathcal{S} \arrow[r] & \g
		\end{tikzcd}$$
		of $\mathcal{I}\longrightarrow\g$ along the inclusion $\mathcal{S}\subset\g$ is called the \textit{universal centralizer} of $G$. One knows that $\mathcal{J}$ is a symplectic subvariety of $T^*G$ \cite{cro-roe:22}. This implies that the group scheme $\mathcal{J}\longrightarrow\mathcal{S}$ is a symplectic groupoid. Since $G_x$ is abelian for all $x\in\greg$ \cite{kos:59}, $\mathcal{J}$ is an abelian symplectic groupoid in the sense of the forthcoming Definition \ref{Definition: Groupoid definition}(i). 
	\end{example}
	
	Let $\G\tto X$ be a symplectic groupoid. Suppose that $\G\tto X$ acts on a symplectic manifold or variety $M$ via a map $\mu:M\longrightarrow X$, called the \textit{moment map}. Following \cite{mik-wei:88}, $M$ is called a \textit{Hamiltonian $\G$-space} if the graph of the $\G$-action is Lagrangian in $\G\times M\times M^{\text{op}}$. In the case that $\G\tto X$ is $T^*G\tto\g^*$ for a Lie group $G$, there is an explicit correspondence between Hamiltonian $G$-spaces and Hamiltonian $T^*G$-spaces \cite{mik-wei:88}. Hamiltonian spaces for symplectic groupoids thereby generalize those for Lie groups.
	
	\subsection{Hamiltonian reduction along pre-Poisson submanifolds and subvarieties}\label{Subsection: Pre-Poisson reduction} Let $\G\tto X$ be a symplectic groupoid. The Lie algebroid of $\G\tto X$ is canonically isomorphic to $T^*X$. One finds that the anchor map is $\sigma^{\vee}:T^*X\longrightarrow TX$ for a Poisson bivector field $\sigma$ on $X$. A submanifold or smooth subvariety $Y\subset X$ is called \textit{pre-Poisson} if $$L_Y\coloneqq (\sigma^{\vee})^{-1}(TY)\cap\mathrm{ann}_{T^*X}(TY)\longrightarrow Y$$ has constant fiber dimension \cite{cat-zam:07,cat-zam:09}. In this case, $L_Y$ is a Lie subalgebroid of $T^*X$. A Lie subgroupoid $\H\tto Y$ of $\G\tto X$ is called a \textit{stabilizer subgroupoid} if $\mathcal{H}$ is isotropic in $\G$ and has $L_Y$ as its Lie algebroid. If $M$ is a Hamiltonian $\G$-space with moment map $\mu:M\longrightarrow X$, then $\mathcal{H}\tto Y$ acts on $\mu^{-1}(Y)$. This observation gives context for the following main result of \cite{cro-may:22}. In the interest of parsimony, we state this result in the smooth and holomorphic categories only; the algebro-geometric version is detailed in \cite[Section 5]{cro-may:22}.
	
	\begin{theorem}[C.--Mayrand]\label{Theorem: Generalized reduction}
		Suppose that $\mathcal{H}\tto Y$ is a stabilizer subgroupoid of $\G\tto X$ over a smooth (resp. complex) pre-Poisson submanifold $Y$ of $X$. Let $(M,\omega)$ be a Hamiltonian $\G$-space with moment map $\mu:M\longrightarrow X$. Consider the inclusion map $j:\mu^{-1}(Y)\longrightarrow M$ and quotient map $\pi:\mu^{-1}(Y)\longrightarrow\mu^{-1}(Y)/\mathcal{H}$. If $\mathcal{H}$ acts freely and properly on $\mu^{-1}(Y)$, then $\mu^{-1}(Y)$ is a smooth (resp. complex) submanifold of $M$, and there exists a unique smooth (resp. complex) manifold structure on $\mu^{-1}(Y)/\mathcal{H}$ for which $\pi$ is a smooth (resp. holomorphic) submersion. In this case, there is a unique smooth (resp. holomorphic) symplectic form $\overline{\omega}$ on $\mu^{-1}(Y)/\mathcal{H}$ satisfying $\pi^*\overline{\omega}=j^*\omega$.
	\end{theorem}
	
	In the algebraic case, one needs the notion of an \textit{algebraic quotient} $\pi:\mu^{-1}(Y)\longrightarrow\mu^{-1}(Y)/\mathcal{H}$ \cite[Definition 5.1]{cro-may:22}. We write $M\sll{Y,\H}\G$ for the symplectic manifold or variety resulting from Theorem \ref{Theorem: Generalized reduction}, and $M\sll{Y,\H,\pi}\G$ for its algebraic counterpart. Source-connected stabilizer subgroupoids $\H\tto Y$ exist in the smooth and holomorphic categories, and all yield the same symplectic manifold $M\sll{Y,\H}\G$. We write $M\sll{Y}\G$ for this symplectic manifold. In the case that $\G$ is the cotangent groupoid of a Lie group $G$, we set $M\sll{Y,\H}G\coloneqq M\sll{Y,\H}(T^*G)$ and $M\sll{Y}G\coloneqq M\sll{Y}(T^*G)$.
	
	\begin{example}[Hamiltonian reduction by a Lie group]
		Suppose that $\G\tto X$ is the cotangent groupoid $T^*G\tto\mathfrak{g}^*$ of a Lie group $G$. If $\xi\in\g^*$, then Example \ref{Example: Cotangent groupoids} implies that $G_{\xi}\longrightarrow\{\xi\}$ is a stabilizer subgroupoid for the pre-Poisson submanifold $\{\xi\}\subset\g^*$. Note that $G_{\xi}$ acts freely and properly on $\mu^{-1}(\xi)$ if and only if this subgroupoid acts freely and properly on $\mu^{-1}(\xi)$. In this case, $M\sll{\{\xi\}}(T^*G)=M\sll{\{\xi\}}G=M\sll{\xi}G$. Theorem \ref{Theorem: Generalized reduction} thereby generalizes the Marsden--Weinstein approach \cite{mar-wei:74} to Hamiltonian reduction.
	\end{example}
	
	\begin{example}[Universal reduced spaces for Hamiltonian Lie group actions]\label{Example: Universal reduced spaces for Hamiltonian actions}
		Let $G$ be a Lie group with Lie algebra $\g$. As explained in Subsection \ref{Subsection: Some identities}, $T^*G$ is a Hamiltonian $G\times G$-space. Restrict the action to one of the subgroup $G=\{e\}\times G\subset G\times G$. In this way, $T^*G$ is a Hamiltonian $T^*G$-space. Suppose that a pre-Poisson submanifold $Y\subset\g^*$ and stabilizer subgroupoid $\H\tto Y$ satisfy the hypotheses of Theorem \ref{Theorem: Generalized reduction} for this $T^*G$-action. Form the symplectic manifold $$\mathfrak{M}_{G,Y,\H}\coloneqq T^*G\sll{Y,\H}G.$$ The action of $G=G\times\{e\}\subset G\times G$ on $T^*G$ induces a Hamiltonian $G$-space structure on $\mathfrak{M}_{G,\mathcal{H},Y}$. On the other hand, let $M$ be a Hamiltonian $G$-space. Suppose that $M$, $Y$, and $\H$ also satisfy the hypotheses of Theorem \ref{Theorem: Generalized reduction}. We then have a canonical symplectomorphism $$M\sll{Y,\H}G\cong (M\times\mathfrak{M}_{G,Y,\H})\sll{0}G,$$
		where $G$ acts diagonally on $M\times \mathfrak{M}_{G,Y,\H}$ \cite[Theorem D(i)]{cro-may:22}. This fact justifies calling $\mathfrak{M}_{G,Y,\H}$ the \textit{universal reduced space} for Hamiltonian reduction by $G$ along $Y$ with respect to $\H$. If $Y=\{0\}$ and $\H$ is $G\longrightarrow\{0\}$, then $\mathfrak{M}_{G,Y,\H}=T^*G$ as Hamiltonian $G$-spaces; here $G=G\times\{e\}\subset G\times G$ acts on $T^*G$ as above.
	\end{example}
	
	\begin{example}[Symplectic implosion]
		Fix the Lie-theoretic data introduced in the first paragraph of Section \ref{Section: Abelianization}. The fundamental Weyl chamber $\mathfrak{t}^*_{+}$ is stratified into pre-Poisson submanifolds of $\g^*$. Given a Hamiltonian $G$-space $M$, this fact allows one to define and realize $M\sll{\mathfrak{t}^*_{+}}G$ as a stratified symplectic space with a stratum-wise Hamiltonian $T$-action. As such, it turns out to coincide with $M_{\text{impl}}$ \cite[Theorem F(ii)]{cro-may:22}. 
	\end{example}
	
	\begin{example}[Pre-images of Poisson transversals under moment maps]\label{Example: Pre-images} Consider a symplectic groupoid $\G\tto X$ and Hamiltonian $\G$-space $\mu:M\longrightarrow X$. A submanifold $S\subset X$ is called a \textit{Poisson transversal} if it has the following property for each symplectic leaf $L\subset X$: $S$ is transverse to $L$ in $X$, and $S\cap L$ is a symplectic submanifold of $L$. In this case, $\mu^{-1}(S)\subset M$ is a symplectic submanifold \cite[Lemma 1(ii)]{fre-mar:17}. On the other hand, $S$ is necessarily a pre-Poisson submanifold of $X$. One finds the trivial groupoid $\H\tto S$ to be a stabilizer subgroupoid for $S$ in $\G$. We have $M\sll{S,\H}\G=\mu^{-1}(S)$ as symplectic manifolds.
		
		The following is an important special case of the above. Let $G$ be a complex semisimple affine algebraic group with Lie algebra $\g$. Suppose that $(e,h,f)\in\g^{\times 3}$ is an $\mathfrak{sl}_2$-triple. The \textit{Slodowy slice} associated to this triple is defined to be the affine subspace $\mathcal{S}\coloneqq e+\g_{f}\subset\g$; it is a Poisson transversal \cite{gan-gin:02}. At the same time, use the left trivialization and Killing form to identify $T^*G$ with $G\times\g$ and $\g^*$ with $\g$. The Hamiltonian action of $G=\{e\}\times G\subset G\times G$ admits
		$$\nu_2:T^*G\longrightarrow\g,\quad (g,x)\mapsto -x$$ as a moment map; see Subsection \ref{Subsection: Some identities}. Noting that $\mathcal{S}\subset\g$ is a Poisson transversal, a straightforward exercise reveals that $-\mathcal{S}\subset\g$ is also a Poisson transversal. It follows that $\nu_2^{-1}(-\mathcal{S})=G\times\mathcal{S}\subset T^*G$ is a symplectic subvariety. We also observe that Hamiltonian action of $G=G\times\{e\}\subset G\times G$ on $T^*G$ restricts to a Hamiltonian $G$-action on $G\times\mathcal{S}$. The Hamiltonian $G$-variety $G\times\mathcal{S}$ features in several recent works \cite{cro-roe:22,cro-ray:19,bie:97,bie:17,cro-van:21}, and is central to the Moore--Tachikawa conjecture \cite{moo-tac:11}. 
	\end{example}
	
	\begin{example}[Reduction by certain abelian group schemes]\label{Example: Reduction by certain abelian group schemes}
		Recall the notation and conventions from Example \ref{Example: Universal centralizer}. This example explains that the universal centralizer $\mathcal{J}\longrightarrow\mathcal{S}$ is an abelian group scheme. Given $n\in\mathbb{Z}_{\geq 1}$, consider the kernel 
		$$\mathcal{J}_n\coloneqq\mathrm{ker}(\underbrace{\mathcal{J}\times_{\mathcal{S}}\cdots\times_{\mathcal{S}}\mathcal{J}}_{n\text{ times}}\longrightarrow\mathcal{J})$$ of the multiplication map from $n$ copies of $\mathcal{J}$; it is a subgroupoid of $T^*G^n=(G\times\g)^n\tto\g^n$ lying over the diagonally embedded copy $\Delta_n\mathcal{S}\subset\mathfrak{g}^n$ of $\mathcal{S}$. It turns out that $\Delta_n\mathcal{S}$ is pre-Poisson in $\g^n$, and that it admits $\mathcal{J}_n\longrightarrow\Delta_n\mathcal{S}$ is a stabilizer subgroupoid of $T^*G^n$ \cite[Theorem 5.10(i)]{cro-may:22}. The symplectic variety $$\mathcal{Z}_n^{\circ}\coloneqq T^*G^n\sll{\Delta_n\mathcal{S},\mathcal{J}_n}G^n$$ exists as a smooth geometric quotient by $\mathcal{J}_n\longrightarrow\Delta_n\mathcal{S}$, where $G^n$ acts on $T^*G^n$ as the subgroup $G^n=\{e\}\times G^n\subset G^n\times G^n$ \cite[Theorem E(ii)]{cro-may:22}. Furthermore, the Hamiltonian action of $G^n=G^n\times\{e\}\subset G^n\times G^n$ on $T^*G^n$ descends to one on $\mathcal{Z}_n^{\circ}$. The varieties $\{\mathcal{Z}_n^{\circ}\}_{n\in\mathbb{Z}_{\geq 1}}$ are the so-called \textit{open Moore--Tachikawa varieties} appearing in unpublished work of Ginzburg--Kazhdan \cite{gin-kaz:24}. This is part of a broader research program on the \textit{Moore--Tachikawa conjecture} \cite{moo-tac:11}. 
	\end{example}
	
	\begin{example}[Subregular semisimple elements]
		Retain the Lie-theoretic notation and conventions in Example \ref{Example: Reduction by certain abelian group schemes}. An element $x\in\g$ is called \textit{subregular} if $\dim\g_x=\mathrm{rk}\hspace{2pt}\g+2$. Consider the loci $$\g_{\text{subreg}}\coloneqq\{x\in\g:\dim\g_x=\mathrm{rk}\hspace{2pt}\g+2\}\quad\text{and}\quad \g_{\mathrm{ss}}\coloneqq\{x\in\g:x\text{ is semisimple}\}.$$ One finds that $\g_{\text{subreg,ss}}\coloneqq \g_{\text{subreg}}\cap\g_{\text{ss}}$ is a pre-Poisson subvariety of $\g$. As per Example \ref{Example: Universal reduced spaces for Hamiltonian actions}, we may consider the universal reduced space $\mathfrak{M}_{G,\g_{\text{subreg,ss}}}$. It turns out that $$\mathfrak{M}_{G,\g_{\text{subreg,ss}}}=\bigsqcup_{x\in\g_{\text{subreg,ss}}}\left(G/[G_x,G_x]\right).$$It is also known that $[G_x,G_x]\cong\operatorname{SL}_2$ for all $x\in\g_{\text{subreg,ss}}$. Further details are available in \cite[Subsection 4.9]{cro-may:22}.
	\end{example}
	
	\begin{example}[Coulomb branches]
		In \cite{gan-web}, Gannon and Webster confirm that certain symplectic varieties are Coulomb branches. The Hamiltonian reduction that these authors perform in \cite[Theorem 1.1]{gan-web} is an instance of the algebro-geometric version of Theorem \ref{Theorem: Generalized reduction}.   
	\end{example}
	
	\subsection{The Moore--Tachikawa conjecture}\label{Subsection: The Moore--Tachikawa conjecture} We now outline the \textit{Moore--Tachikawa conjecture} \cite{moo-tac:11} referenced in Example \ref{Example: Reduction by certain abelian group schemes}. A first step is to define a symmetric monoidal category $\mathrm{MT}$. The objects of $\mathrm{MT}$ are complex semisimple affine algebraic groups. Suppose that $G$ and $H$ are two such groups. Consider a possibly singular affine Poisson variety $M$ that induces a symplectic structure on an open dense subset of its smooth locus. Assume that $M$ carries an algebraic action of $G\times H$, in which the subgroup $G=G\times\{e\}\subset G\times H$ (resp. $H=\{e\}\times H\subset G\times H$) acts in a Hamiltonian fashion on $M$ (resp. $M^{\text{op}}$). One defines $\mathrm{Hom}_{\mathrm{MT}}(G,H)$ to consist of such varieties $M$, up to $G\times H$-equivariant Poisson isomorphisms that intertwine moment maps.
	
	To define morphism composition in $\mathrm{MT}$, let $G$, $H$, and $K$ be complex semisimple affine algebraic groups. Suppose that $[M]\in\mathrm{Hom}_{\mathrm{MT}}(G,H)$ and $[N]\in\mathrm{Hom}_{\mathrm{MT}}(H,K)$. The diagonal action of $H$ on $M\times N$ is Hamiltonian with respect to $M^{\text{op}}\times N$. Observe that the affine GIT quotient $(M^{\text{op}}\times N)\sll{0}H$ inherits a residual Hamiltonian action of $G\times K$. One defines composition in $\mathrm{MT}$ by \begin{equation}\label{Equation: Composition}[N]\circ [M]\coloneqq [(M^{\text{op}}\times N)\sll{0}H]\in\mathrm{Hom}_{\mathrm{MT}}(G,K).\end{equation}
	Direct products of objects and morphisms constitute the tensor product of a symmetric monoidal structure on $\mathrm{MT}$. 
	
	We now let $\mathrm{COB}_2$ denote the category of $2$-dimensional cobordisms. The objects of $\mathrm{COB}_2$ are compact smooth $1$-manifolds. A morphism from an object $M$ to an object $N$ in this category is a $2$-manifold with boundary $M\sqcup N$, up to appropriate topological equivalence. Morphism composition is defined by concatenation along common boundaries. The symmetric monoidal structure on $\mathrm{MT}$ is given by disjoint unions of objects and morphisms. A \textit{2-dimensional topological quantum field theory (TQFT)} valued in a symmetric monoidal category $\mathcal{C}$ is a symmetric monoidal functor $\eta:\mathrm{COB}_2\longrightarrow\mathcal{C}$. The Moore--Tachikawa conjecture concerns $2$-dimensional TQFTs valued in $\mathrm{MT}$.
	
	\begin{conjecture}[Moore--Tachikawa]\label{Conjecture: Moore--Tachikawa}
		Let $G$ be a connected complex semisimple affine algebraic group with Lie algebra $\g$. Suppose that $\mathcal{S}\coloneqq e+\g_f$ is the Slodowy slice associated to a principal $\mathfrak{sl}_2$-triple $(e,h,f)\in\g^{\times 3}$. There exists a $2$-dimensional TQFT $\eta_G:\mathrm{COB}_2\longrightarrow\mathrm{MT}$ satisfying $\eta_G(S^1)=G$ and $\eta_G(\begin{tikzpicture}[
			baseline=-2.5pt,
			every tqft/.append style={
				transform shape, rotate=90, tqft/circle x radius=4pt,
				tqft/circle y radius= 2pt,
				tqft/boundary separation=0.6cm, 
				tqft/view from=incoming,
			}
			]
			\pic[
			tqft/cup,
			name=d,
			every incoming lower boundary component/.style={draw},
			every outgoing lower boundary component/.style={draw},
			every incoming boundary component/.style={draw},
			every outgoing boundary component/.style={draw},
			cobordism edge/.style={draw},
			cobordism height= 1cm,
			];
		\end{tikzpicture})=[G\times\mathcal{S}]$,
		where $G\times\mathcal{S}$ is equipped with the Hamiltonian $G$-variety structure in Example \ref{Example: Pre-images}. 
	\end{conjecture}
	
	This conjecture has received considerable attention in the last decade. In particular, it is known to hold if $G$ is simple of Lie type $A$ \cite{gin-kaz:24,bra-fin-nak:19}. Ginzburg and Kazhdan attempt a Lie type-independent approach in unpublished work \cite{gin-kaz:24}. To this end, recall the Hamiltonian $G^n$-varieties $\mathcal{Z}_n^{\circ}$ in Subsection \ref{Example: Reduction by certain abelian group schemes}. The affinization $\mathcal{Z}_{n}$ of $\mathcal{Z}_{n}^{\circ}$ is an affine Poisson scheme carrying a Hamiltonian $G^n$-action. It turns out that $\mathcal{Z}_n$ is a morphism in $\mathrm{MT}$ if and only if $\mathcal{Z}_n$ is of finite type. If the schemes $\{\mathcal{Z}_n\}_{n\in\mathbb{Z}_{\geq 1}}$ are of finite type, then arguments of Ginzburg--Kazhdan imply that these schemes satisfy necessary conditions to yield a TQFT $\eta_G:\mathrm{COB}_2\longrightarrow\mathrm{MT}$ with $\eta_G(S^1)=G$ and $\eta_G(\begin{tikzpicture}[
		baseline=-2.5pt,
		every tqft/.append style={
			transform shape, rotate=90, tqft/circle x radius=4pt,
			tqft/circle y radius= 2pt,
			tqft/boundary separation=0.6cm, 
			tqft/view from=incoming,
		}
		]
		\pic[
		tqft/cup,
		name=d,
		every incoming lower boundary component/.style={draw},
		every outgoing lower boundary component/.style={draw},
		every incoming boundary component/.style={draw},
		every outgoing boundary component/.style={draw},
		cobordism edge/.style={draw},
		cobordism height= 1cm,
		];
	\end{tikzpicture})=[G\times\mathcal{S}]$.
	
	\subsection{A scheme-theoretic version}\label{Subsection: A scheme-theoretic version}
	We now outline a scheme-theoretic version of Theorem \ref{Theorem: Generalized reduction}, as developed in \cite{cro-may:241}. While our primary motivation is to further investigate Conjecture \ref{Conjecture: Moore--Tachikawa}, we witness applications to results of \'{S}niatycki--Weinstein \cite{sni-wei:83}.
	
	Suppose that $\G\tto X$ is a groupoid object in the category of affine varieties over $\Bbbk=\mathbb{R}$ or $\Bbbk=\mathbb{C}$. We call $\G$ an \textit{affine algebraic groupoid} if $\G$ and $X$ are smooth as varieties, and the source and target maps are smooth as morphisms of varieties. An affine scheme $\mu:M\longrightarrow X$ over $X$ is called an \textit{affine $\G$-scheme} if it carries an algebraic action $\G\fp{\sss}{\mu}M\longrightarrow M$ of $\G$. We write $\Bbbk[M]$ for the algebra of global sections of $M$, and $\Bbbk[M]^{\G}\subset\Bbbk[M]$ for the subalgebra of $\G$-invariant elements.
	
	Let $\G\tto X$ be an affine algebraic groupoid. We call $\G$ an \textit{affine symplectic groupoid} if $\G$ carries an algebraic symplectic form for which the graph of groupoid multiplication is coisotropic in $\G\times\G\times\G^{\text{op}}$. The notions of \textit{pre-Poisson subvariety} of $X$ and \textit{stabilizer subgroupoid} of $\G$ are analogous to those discussed in Subsection \ref{Subsection: Pre-Poisson reduction}. At the same time, consider an affine $\G$-scheme $\mu:M\longrightarrow X$ over $X$. We call $M$ an \textit{affine Hamiltonian $\G$-scheme} if $M$ is Poisson and the graph of the $\G$-action is coisotropic in $\G\times M\times M^{\text{op}}$. If $\H\tto Y$ is a stabilizer subgroupoid of $\G$ over a closed, pre-Poisson subvariety $Y\subset X$, then the closed subscheme scheme $\mu^{-1}(Y)\subset X$ is an affine Hamiltonian $\H$-scheme. The following is proved in \cite{cro-may:241}.
	
	\begin{theorem}[C.--Mayrand]\label{Theorem: Pure algebra}
		Consider an affine symplectic groupoid $\G\tto X$ and affine Hamiltonian $\G$-scheme $\mu:M\longrightarrow X$. Let $\H\tto Y$ be a stabilizer subgroupoid of $\G$ over a closed, pre-Poisson subvariety $Y\subset X$. Write $j:\mu^{-1}(Y)\longrightarrow M$ for the inclusion morphism. There exists a unique Poisson bracket on $\Bbbk[\mu^{-1}(S)]^{\mathcal{H}}$ satisfying $\{j^*f_1,j^*f_2\}=j^*\{f_1,f_2\}$ for all $f_1,f_2\in\Bbbk[M]$ with $j^*f_1,j^*f_2\in\Bbbk[\mu^{-1}(S)]^{\mathcal{H}}$.
	\end{theorem}
	
	If the hypotheses of Theorem \ref{Theorem: Pure algebra} are satisfied, then $$M\sll{Y,\H}\G\coloneqq\mathrm{Spec}(\Bbbk[\mu^{-1}(S)]^{\mathcal{H}})$$ is an affine Poisson scheme.
	
	\begin{example}[Scheme-theoretic \'{S}niatycki--Weinstein reduction]
		Let $G$ be a connected real Lie group with Lie algebra $\g$. Consider a Hamiltonian $G$-space $M$ with moment map $\mu:M\longrightarrow\g^*$. Given $\xi\in\g$, write $\mu^{\xi}:M\longrightarrow\mathbb{R}$ for the smooth map
		$$\mu^{\xi}:M\longrightarrow\mathbb{R},\quad \mu^{\xi}(p)\coloneqq \mu(p)(\xi),\quad p\in M.$$
		Denote by $I_{\mu}\subset C^{\infty}(M)$ the ideal generated by $\mu^{\xi}$ for all $\xi\in\g$. Note that $I_{\mu}$ is invariant under the action of $G$ on $C^{\infty}(M)$. It follows that $G$ acts on $C^{\infty}(M)/I_{\mu}$. At the same time, consider the quotient map $\pi:C^{\infty}(M)\longrightarrow C^{\infty}(M)/I_{\mu}$. \'{S}niatycki and Weinstein show that there exists a unique Poisson bracket on $(C^{\infty}(M)/I_{\mu})^G$ such that $\{\pi(f_1),\pi(f_2)\}=\pi(\{f_1,f_2\})$ for all $f_1,f_2\in C^{\infty}(M)$ with $\pi(f_1),\pi(f_2)\in (C^{\infty}(M)/I_{\mu})^G$ \cite{sni-wei:83}. In the case of a free and proper $G$-action on $\mu^{-1}(0)$, the Marsden--Weinstein symplectic structure on $\mu^{-1}(0)/G$ is encoded in a Poisson algebra structure on $C^{\infty}(\mu^{-1}(0)/G)$. The \'{S}niatycki--Weinstein Poisson algebra $(C^{\infty}(M)/I_{0})^G$ is canonically isomorphic to $C^{\infty}(\mu^{-1}(0)/G)$ in this case.
		
		Let $G$ be an affine algebraic group over $\mathbb{R}$ or $\mathbb{C}$. Consider an affine Poisson scheme with a Hamiltonian action of $G$ and moment map $\mu:M\longrightarrow\g^*$. As one might imagine, $M$ is an affine Hamiltonian $T^*G$-scheme. We may therefore apply Theorem \ref{Theorem: Pure algebra} to $M$ with $\G=T^*G$, $S=\{0\}\subset\g^*$, and $\H$ the isotropy group of $0$ in $T^*G$. This yields a scheme-theoretic version of the previous paragraph. 
	\end{example}
	
	\subsection{The Moore--Tachikawa conjecture via symplectic groupoids}\label{Subsection: MT via SG}
	It remains to explain the importance of Subsection \ref{Subsection: A scheme-theoretic version} to the Moore--Tachikawa conjecture. We set $\Bbbk=\mathbb{C}$, and define a symmetric monoidal category $\mathrm{AMT}$ as follows. Objects of $\mathrm{AMT}$ are affine symplectic groupoids. If $\G$ and $\H$ are affine symplectic groupoids, then $\mathrm{Hom}_{\mathrm{AMT}}(\G,\H)$ consists of affine Hamiltonian Poisson $\G\times\H^{\text{op}}$-schemes, up to equivariant, moment map-preserving Poisson isomorphisms. Morphism composition in $\mathrm{AMT}$ is defined via Theorem \ref{Theorem: Pure algebra}; see \cite[Subsection 4.3]{cro-may:241} for more details. By defining the tensor product to be the direct product of objects and morphisms, one renders $\mathrm{AMT}$ a symmetric monoidal category.
	
	Consider the Hamiltonian $G$-variety $G\times\mathcal{S}$ in Example \ref{Example: Pre-images}. Note that
	$G\times\mathcal{S}=\ttt^{-1}(\mathcal{S})$ for the target $\ttt$ of the cotangent groupoid $T^*G\tto\g$. Let us now consider an arbitrary affine symplectic groupoid $\G\tto X$. In the spirit of Conjecture \ref{Conjecture: Moore--Tachikawa}, one might seek a TQFT $\eta_{\G}:\mathrm{COB}_2\longrightarrow\mathrm{AMT}$ satisfying $\eta_{\G}(S^1)=\G$ and $\eta_{\G}(\begin{tikzpicture}[
		baseline=-2.5pt,
		every tqft/.append style={
			transform shape, rotate=90, tqft/circle x radius=4pt,
			tqft/circle y radius= 2pt,
			tqft/boundary separation=0.6cm, 
			tqft/view from=incoming,
		}
		]
		\pic[
		tqft/cup,
		name=d,
		every incoming lower boundary component/.style={draw},
		every outgoing lower boundary component/.style={draw},
		every incoming boundary component/.style={draw},
		every outgoing boundary component/.style={draw},
		cobordism edge/.style={draw},
		cobordism height= 1cm,
		];
	\end{tikzpicture})=[\ttt^{-1}(\mathcal{S})]$ for a ``suitable" subvariety $\mathcal{S}\subset X$. One might call this the \textit{Moore--Tachikawa conjecture in $\mathrm{AMT}$}. In the interest of being more precise, we make the following definitions. Let $\G\big\vert_{U}\tto U$ denote the restriction $\sss^{-1}(U)\cap\ttt^{-1}(U)\tto U$ of $\G\tto X$ to $U\subset X$, and note that $\G_{x}\coloneqq\G_{\{x\}}\tto\{x\}$ is the \textit{isotropy group of $x$} for all $x\in X$.
	
	\begin{definition}\label{Definition: Groupoid definition}
		Let $\G\tto X$ be an algebraic symplectic groupoid.
		\begin{itemize}
			\item[\textup{(i)}] We call $\G$ \textit{abelian} if $\sss=\ttt$ and $\G_x$ is abelian for all $x\in X$.
			\item[\textup{(ii)}] We call $\G$ \textit{abelianizable} if it is Morita equivalent to an abelian affine symplectic groupoid.
			\item[\textup{(iii)}] We call $\G$ \textit{Hartogs abelianizable} there exists an open subset $U\subset X$ such that $\G\big\vert_{U}\tto U$ is abelianizable and $X\setminus U$ has codimension at least two in $X$.
			\item[\textup{(iv)}] A smooth closed subvariety $\mathcal{S}\subset X$ is called a \textit{global slice} if it intersects each $\G$-orbit in $X$ transversely in a singleton, and $\G_x$ is abelian for all $x\in\mathcal{S}$. 
			\item[\textup{(v)}] A smooth closed subvariety $\mathcal{S}\subset X$ is called a \textit{Hartogs slice} if there exists an open subset $U\subset X$ such that $\mathcal{S}$ is a global slice to $\G\big\vert_U\tto U$ and $X\setminus U$ has codimension at least two in $X$.
		\end{itemize}
	\end{definition}
	
	\begin{example}\label{Example: Adjoint}
		Suppose that $\g$ is a finite-dimensional complex semisimple Lie algebra with adjoint group $G$. Let $\G\tto X$ be the cotangent groupoid $T^*G\tto\g^*$. Use the Killing form and left trivialization to identify $T^*G$ and $\g^*$ with $G\times\g$ and $\g$, respectively. Consider the open, dense, $G$-invariant subset
		$$\greg\coloneqq\{x\in\g:\dim\g_x=\mathrm{rk}\hspace{2pt}\g\}\subset\g.$$ Write $(T^*G)_{\text{reg}}\tto\greg$ for the pullback of $T^*G\tto\g$ along the inclusion $\greg\subset\g$. On the other hand, let $\mathcal{S}=e+\g_f$ be the Slodowy slice associated to a principal $\mathfrak{sl}_2$-triple $e+\g_f\subset\g$. Results of Kostant imply that $\mathcal{S}$ is a global slice to $(T^*G)_{\text{reg}}\tto\greg$ \cite{kos:59,kos:63}. It is also known that $\g\setminus\greg$ has codimension $3$ in $\g$ \cite[Lemma 3.6]{pop:08}. This implies that $\mathcal{S}$ is a Hartogs slice to $T^*G\tto\g$.
		
		Recall the universal centralizer $\mathcal{J}\longrightarrow\mathcal{S}$ discussed in Example \ref{Example: Reduction by certain abelian group schemes}. In this example, we explain that $\mathcal{J}\longrightarrow\mathcal{S}$ is an abelian affine symplectic groupoid. It turns out that $(T^*G)_{\text{reg}}\tto\greg$ is Morita equivalent to $\mathcal{J}\longrightarrow\mathcal{S}$. In particular, $(T^*G)_{\text{reg}}\tto\greg$ is abelianizable. The fact that $\g\setminus\greg$ has codimension $3$ in $\g$ implies that $T^*G\tto\g$ is Hartogs abelianizable.
	\end{example}
	
	Let $\G\tto X$ be an algebraic symplectic groupoid. If this groupoid admits a global (resp. Hartogs global) slice, then it is abelianizable (resp. Hartogs abelianizable).
	
	\begin{theorem}[C.--Mayrand]\label{Theorem: Nice}
		Let $\G\tto X$ be an affine symplectic groupoid.
		\begin{itemize}
			\item[\textup{(i)}] If $\G\tto X$ is Hartogs abelianizable, then it determines a canonical TQFT $\eta_{\G}:\mathrm{COB}_2\longrightarrow\mathrm{AMT}$ satisfying $\eta_{\G}(S^1)=\G$.
			\item[\textup{(ii)}] If $\mathcal{S}\subset X$ is a Hartogs slice, then $\eta_{\G}(\begin{tikzpicture}[
				baseline=-2.5pt,
				every tqft/.append style={
					transform shape, rotate=90, tqft/circle x radius=4pt,
					tqft/circle y radius= 2pt,
					tqft/boundary separation=0.6cm, 
					tqft/view from=incoming,
				}
				]
				\pic[
				tqft/cup,
				name=d,
				every incoming lower boundary component/.style={draw},
				every outgoing lower boundary component/.style={draw},
				every incoming boundary component/.style={draw},
				every outgoing boundary component/.style={draw},
				cobordism edge/.style={draw},
				cobordism height= 1cm,
				];
			\end{tikzpicture})=[\ttt^{-1}(\mathcal{S})]$.
		\end{itemize}
	\end{theorem}
	
	\begin{example}[The Moore--Tachikawa conjecture in $\mathrm{AMT}$]
		Retain the objects and conventions in Example \ref{Example: Adjoint}. The TQFT $\eta_{T^*G}:\mathrm{COB}_2\longrightarrow\mathrm{AMT}$ satisfies $\eta_{T^*G}(S^1)=T^*G$ and $\eta_{T^*G}(\begin{tikzpicture}[
			baseline=-2.5pt,
			every tqft/.append style={
				transform shape, rotate=90, tqft/circle x radius=4pt,
				tqft/circle y radius= 2pt,
				tqft/boundary separation=0.6cm, 
				tqft/view from=incoming,
			}
			]
			\pic[
			tqft/cup,
			name=d,
			every incoming lower boundary component/.style={draw},
			every outgoing lower boundary component/.style={draw},
			every incoming boundary component/.style={draw},
			every outgoing boundary component/.style={draw},
			cobordism edge/.style={draw},
			cobordism height= 1cm,
			];
		\end{tikzpicture})=[G\times\mathcal{S}]$. This proves the Moore--Tachikawa conjecture in $\AMT$. To the authors' knowledge, the original conjecture remains open outside of Lie type $A$. 
	\end{example}
	
	\begin{example}[More general Slodowy slices]
		Retain the objects and conventions in Example \ref{Example: Pre-images}. Let $(e,h,f)\in\g^{\times 3}$ be an arbitrary $\mathfrak{sl}_2$-triple. Example \ref{Example: Pre-images} explains that the Slodowy slice $\mathcal{S}\coloneqq e+\g_f$ is a Poisson transversal in $\g$. As such, $\mathcal{S}$ is a Poisson variety. The restriction $(T^*G)\big\vert_{\mathcal{S}}\tto\mathcal{S}$ is a symplectic groupoid integrating this Poisson variety.
		
		Let us assume that $\g=\mathfrak{sl}_n$. Suppose that $e,f\in\mathfrak{sl}_n$ belong to the minimal nilpotent orbit in $\mathfrak{sl}_n$. The symplectic groupoid $(T^*G)\big\vert_{\mathcal{S}}\tto\mathcal{S}$ turns out to admit a Hartogs slice; see \cite[Subsection 7.3]{cro-may:24}.
	\end{example}
	
	\begin{example}[Moore--Tachikawa groups]
		Let $G$ be an affine algebraic group with Lie algebra $\g$. Consider the open subset
		$$\g^*_{\text{reg}}\coloneqq\{\xi\in\g^*:\dim\g_{\xi}\leq\dim\g_{\eta}\text{ for all }\eta\in\g^*\}.$$
		We call $G$ a \textit{Moore--Tachikawa group} if it satisfies the following conditions:
		\begin{itemize}
			\item[\textup{(i)}] $\g^*\setminus\g^*_{\text{reg}}$ has codimension at least $2$ in $\g^*$; 
			\item[\textup{(ii)}] $(T^*G)\big\vert_{\g^*_{\text{reg}}}\tto\g^*_{\text{reg}}$ is abelianizable.
		\end{itemize}
		Theorem \ref{Theorem: Nice} tells us that a Moore--Tachikawa group $G$ induces a TQFT $\eta_{T^*G}:\mathrm{COB}_2\longrightarrow\mathrm{AMT}$. While Kostant's results imply that reductive groups are Moore--Tachikawa, there are non-reductive example of Moore--Tachikawa groups; see Subsections 8.1 and 8.2 of \cite{cro-may:24} for more details.  
	\end{example}
	
	\bibliographystyle{acm}
	\bibliography{Hamiltonian}

\begin{thebibliography}{10}

\bibitem{alb:89}
{\sc Albert, C.}
\newblock Le th\'{e}or\`eme de r\'{e}duction de {M}arsden-{W}einstein en
  g\'{e}om\'{e}trie cosymplectique et de contact.
\newblock {\em J. Geom. Phys. 6}, 4 (1989), 627--649.

\bibitem{ale-kos-mei}
{\sc Alekseev, A., Kosmann-Schwarzbach, Y., and Meinrenken, E.}
\newblock Quasi-{P}oisson manifolds.
\newblock {\em Canad. J. Math. 54}, 1 (2002), 3--29.

\bibitem{ale-mal-mei}
{\sc Alekseev, A., Malkin, A., and Meinrenken, E.}
\newblock Lie group valued moment maps.
\newblock {\em J. Differential Geom. 48}, 3 (1998), 445--495.

\bibitem{ati-bot:84}
{\sc Atiyah, M.~F., and Bott, R.}
\newblock The {Y}ang-{M}ills equations over {R}iemann surfaces.
\newblock {\em Philos. Trans. Roy. Soc. London Ser. A 308}, 1505 (1983),
  523--615.

\bibitem{bie:97}
{\sc Bielawski, R.}
\newblock Hyperk\"{a}hler structures and group actions.
\newblock {\em J. London Math. Soc. (2) 55}, 2 (1997), 400--414.

\bibitem{bie:17}
{\sc Bielawski, R.}
\newblock Slices to sums of adjoint orbits, the {A}tiyah-{H}itchin manifold,
  and {H}ilbert schemes of points.
\newblock {\em Complex Manifolds 4}, 1 (2017), 16--36.

\bibitem{bra-fin-nak:19}
{\sc Braverman, A., Finkelberg, M., and Nakajima, H.}
\newblock Ring objects in the equivariant derived {S}atake category arising
  from {C}oulomb branches.
\newblock {\em Adv. Theor. Math. Phys. 23}, 2 (2019), 253--344.
\newblock Appendix by Gus Lonergan.

\bibitem{bal:23}
{\sc B\u{a}libanu, A.}
\newblock The partial compactification of the universal centralizer.
\newblock {\em Selecta Math. (N.S.) 29}, 5 (2023), Paper No. 85, 36.

\bibitem{bal-may:22}
{\sc B\u{a}libanu, A., and Mayrand, M.}
\newblock Reduction along strong {D}irac maps.
\newblock arXiv:2210.07200.

\bibitem{cat-zam:07}
{\sc Cattaneo, A.~S., and Zambon, M.}
\newblock Pre-{P}oisson submanifolds.
\newblock In {\em Travaux math\'{e}matiques. {V}ol. {XVII}}, vol.~17 of {\em
  Trav. Math.} Fac. Sci. Technol. Commun. Univ. Luxemb., Luxembourg, 2007,
  pp.~61--74.

\bibitem{cat-zam:09}
{\sc Cattaneo, A.~S., and Zambon, M.}
\newblock Coisotropic embeddings in {P}oisson manifolds.
\newblock {\em Trans. Amer. Math. Soc. 361}, 7 (2009), 3721--3746.

\bibitem{cro-may:22}
{\sc Crooks, P., and Mayrand, M.}
\newblock Symplectic reduction along a submanifold.
\newblock {\em Compos. Math. 158}, 9 (2022), 1878--1934.

\bibitem{cro-may:24}
{\sc Crooks, P., and Mayrand, M.}
\newblock The {M}oore--{T}achikawa conjecture via shifted symplectic geometry.
\newblock arXiv:2409.03532.

\bibitem{cro-may:241}
{\sc Crooks, P., and Mayrand, M.}
\newblock Scheme-theoretic coisotropic reduction.
\newblock arXiv:2408.11932.

\bibitem{cro-ray:19}
{\sc Crooks, P., and Rayan, S.}
\newblock Abstract integrable systems on hyperk\"{a}hler manifolds arising from
  {S}lodowy slices.
\newblock {\em Math. Res. Lett. 26}, 1 (2019), 9--33.

\bibitem{cro-roe:22}
{\sc Crooks, P., and R\"{o}ser, M.}
\newblock The log symplectic geometry of {P}oisson slices.
\newblock {\em J. Symplectic Geom. 20}, 1 (2022), 135--189.

\bibitem{cro-van:21}
{\sc Crooks, P., and van Pruijssen, M.}
\newblock An application of spherical geometry to hyperk\"{a}hler slices.
\newblock {\em Canad. J. Math. 73}, 3 (2021), 687--716.

\bibitem{cro-wei:23}
{\sc Crooks, P., and Weitsman, J.}
\newblock Abelianization and the {D}uistermaat-{H}eckman theorem.
\newblock {\em Bull. Lond. Math. Soc. 55}, 6 (2023), 2732--2742.

\bibitem{cro-wei:232}
{\sc Crooks, P., and Weitsman, J.}
\newblock The double {G}elfand-{C}etlin system, invariance of polarization, and
  the {P}eter-{W}eyl theorem.
\newblock {\em J. Geom. Phys. 194\/} (2023), Paper No. 105011, 12.

\bibitem{cro-wei:24}
{\sc Crooks, P., and Weitsman, J.}
\newblock Gelfand-{C}etlin abelianizations of symplectic quotients.
\newblock {\em Pacific J. Math. 333}, 2 (2024), 253--271.

\bibitem{fre-mar:17}
{\sc Frejlich, P., and M\u{a}rcu\c{t}, I.}
\newblock The normal form theorem around {P}oisson transversals.
\newblock {\em Pacific J. Math. 287}, 2 (2017), 371--391.

\bibitem{gan-gin:02}
{\sc Gan, W.~L., and Ginzburg, V.}
\newblock Quantization of {S}lodowy slices.
\newblock {\em Int. Math. Res. Not.}, 5 (2002), 243--255.

\bibitem{gan-gin:04}
{\sc Gan, W.~L., and Ginzburg, V.}
\newblock Hamiltonian reduction and {M}aurer-{C}artan equations.
\newblock {\em Mosc. Math. J. 4}, 3 (2004), 719--727, 784.

\bibitem{gan-web}
{\sc Gannon, T., and Webster, B.}
\newblock Functoriality of {C}oulomb branches.
\newblock arXiv:2501.09962.

\bibitem{gin-kaz:24}
{\sc Ginzburg, V., and Kazhdan, D.}
\newblock Algebraic symplectic manifolds arising in {S}icilian theories.
\newblock {\em Private communication\/}.

\bibitem{gui-jef-sja:02}
{\sc Guillemin, V., Jeffrey, L., and Sjamaar, R.}
\newblock Symplectic implosion.
\newblock {\em Transform. Groups 7}, 2 (2002), 155--184.

\bibitem{gui-ste:83}
{\sc Guillemin, V., and Sternberg, S.}
\newblock The {G}elfand-{C}etlin system and quantization of the complex flag
  manifolds.
\newblock {\em J. Funct. Anal. 52}, 1 (1983), 106--128.

\bibitem{har-kav:15}
{\sc Harada, M., and Kaveh, K.}
\newblock Integrable systems, toric degenerations and {O}kounkov bodies.
\newblock {\em Invent. Math. 202}, 3 (2015), 927--985.

\bibitem{her-sch-sea:20}
{\sc Herbig, H.-C., Schwarz, G.~W., and Seaton, C.}
\newblock Symplectic quotients have symplectic singularities.
\newblock {\em Compos. Math. 156}, 3 (2020), 613--646.

\bibitem{hklr:87}
{\sc Hitchin, N.~J., Karlhede, A., Lindstr\"{o}m, U., and Ro\v{c}ek, M.}
\newblock Hyper-{K}\"{a}hler metrics and supersymmetry.
\newblock {\em Comm. Math. Phys. 108}, 4 (1987), 535--589.

\bibitem{hof-lan:23}
{\sc Hoffman, B., and Lane, J.}
\newblock Stratified gradient {H}amiltonian vector fields and collective
  integrable systems.
\newblock arXiv:2008.13656.

\bibitem{hur-jef-sja}
{\sc Hurtubise, J., Jeffrey, L., and Sjamaar, R.}
\newblock Group-valued implosion and parabolic structures.
\newblock {\em Amer. J. Math. 128}, 1 (2006), 167--214.

\bibitem{kir:84}
{\sc Kirwan, F.~C.}
\newblock {\em Cohomology of quotients in symplectic and algebraic geometry},
  vol.~31 of {\em Mathematical Notes}.
\newblock Princeton University Press, Princeton, NJ, 1984.

\bibitem{kos:59}
{\sc Kostant, B.}
\newblock The principal three-dimensional subgroup and the {B}etti numbers of a
  complex simple {L}ie group.
\newblock {\em Amer. J. Math. 81\/} (1959), 973--1032.

\bibitem{kos:63}
{\sc Kostant, B.}
\newblock Lie group representations on polynomial rings.
\newblock {\em Amer. J. Math. 85\/} (1963), 327--404.

\bibitem{ler:95}
{\sc Lerman, E.}
\newblock Symplectic cuts.
\newblock {\em Math. Res. Lett. 2}, 3 (1995), 247--258.

\bibitem{mar-wei:74}
{\sc Marsden, J., and Weinstein, A.}
\newblock Reduction of symplectic manifolds with symmetry.
\newblock {\em Rep. Mathematical Phys. 5}, 1 (1974), 121--130.

\bibitem{mar-rat:86}
{\sc Marsden, J.~E., and Ratiu, T.}
\newblock Reduction of {P}oisson manifolds.
\newblock {\em Lett. Math. Phys. 11}, 2 (1986), 161--169.

\bibitem{mei:96}
{\sc Meinrenken, E.}
\newblock On {R}iemann-{R}och formulas for multiplicities.
\newblock {\em J. Amer. Math. Soc. 9}, 2 (1996), 373--389.

\bibitem{mei-sja}
{\sc Meinrenken, E., and Sjamaar, R.}
\newblock Singular reduction and quantization.
\newblock {\em Topology 38}, 4 (1999), 699--762.

\bibitem{mey:73}
{\sc Meyer, K.~R.}
\newblock Symmetries and integrals in mechanics.
\newblock In {\em Dynamical systems ({P}roc. {S}ympos., {U}niv. {B}ahia,
  {S}alvador, 1971)\/} (1973), Academic Press, New York-London, pp.~259--272.

\bibitem{mik-wei:88}
{\sc Mikami, K., and Weinstein, A.}
\newblock Moments and reduction for symplectic groupoids.
\newblock {\em Publ. Res. Inst. Math. Sci. 24}, 1 (1988), 121--140.

\bibitem{moo-tac:11}
{\sc Moore, G.~W., and Tachikawa, Y.}
\newblock On 2d {TQFT}s whose values are holomorphic symplectic varieties.
\newblock In {\em String-{M}ath 2011}, vol.~85 of {\em Proc. Sympos. Pure
  Math.} Amer. Math. Soc., Providence, RI, 2012, pp.~191--207.

\bibitem{pan-toe-vaq-vez:13}
{\sc Pantev, T., To\"{e}n, B., Vaqui\'{e}, M., and Vezzosi, G.}
\newblock Shifted symplectic structures.
\newblock {\em Publ. Math. Inst. Hautes \'{E}tudes Sci. 117\/} (2013),
  271--328.

\bibitem{pop:08}
{\sc Popov, V.~L.}
\newblock Irregular and singular loci of commuting varieties.
\newblock {\em Transform. Groups 13}, 3-4 (2008), 819--837.

\bibitem{saf:17}
{\sc Safronov, P.}
\newblock Symplectic implosion and the {G}rothendieck-{S}pringer resolution.
\newblock {\em Transform. Groups 22}, 3 (2017), 767--792.

\bibitem{sja:95}
{\sc Sjamaar, R.}
\newblock Holomorphic slices, symplectic reduction and multiplicities of
  representations.
\newblock {\em Ann. of Math. (2) 141}, 1 (1995), 87--129.

\bibitem{sja-ler:91}
{\sc Sjamaar, R., and Lerman, E.}
\newblock Stratified symplectic spaces and reduction.
\newblock {\em Ann. of Math. (2) 134}, 2 (1991), 375--422.

\bibitem{sni-wei:83}
{\sc \'{S}niatycki, J.~e., and Weinstein, A.}
\newblock Reduction and quantization for singular momentum mappings.
\newblock {\em Lett. Math. Phys. 7}, 2 (1983), 155--161.

\end{thebibliography}
\end{document}